\documentclass[review]{elsarticle}

\usepackage{lineno}
\modulolinenumbers[5]
\usepackage{color}
\usepackage{graphicx}
\graphicspath{{figure/}}
\usepackage[normalem]{ulem}
\let\isout\sout \renewcommand{\sout}[1]{\ifmmode\text{\isout{\ensuremath{{\color{blue}#1}}}}\else\isout{{\color{blue}#1}}\fi}
\usepackage{hyperref}
\usepackage{multirow}
\usepackage{amsmath,amstext,amssymb}
\usepackage{amsthm}
\usepackage{tabularx,longtable,multirow,diagbox}
\usepackage{enumerate}
\usepackage{mathrsfs}
\usepackage{caption}
\usepackage{booktabs} 
\usepackage{geometry}
\usepackage{subfigure}
\usepackage{algorithm}
\usepackage{algorithmicx}
\usepackage{algpseudocode}

\newcommand{\norm}[1]{\left\Vert#1\right\Vert}
\newcommand{\normx}[1]{\Vert#1\Vert}
\newcommand{\normxsz}[1]{\Vert#1\Vert_{\Sigma_0}}
\newcommand{\normxst}[1]{\Vert#1\Vert_{\Sigma_T}}
\newcommand{\normxss}[1]{\Vert#1\Vert_{\Sigma_s}}
\newcommand{\snorm}[1]{\left|#1\right|}
\newcommand{\snormgh}[1]{\left|#1\right|_{g_h}}

\newcommand{\norme}[1]{\left\Vert\hskip -0.8pt \left\vert #1 \right\vert\hskip -0.8pt\right\Vert}
\newcommand{\normeh}[1]{\left\Vert\hskip -0.8pt \left\vert #1 \right\vert\hskip -0.8pt\right\Vert_h}
\newcommand{\normes}[1]{\left\Vert\hskip -0.8pt \left\vert #1 \right\vert\hskip -0.8pt\right\Vert_*}
\newcommand{\normehs}[1]{\left\Vert\hskip -0.8pt \left\vert #1 \right\vert\hskip -0.8pt\right\Vert_{h,*}}
\newcommand{\pd}[1]{\left\langle #1\right\rangle}

\newcommand{\bk}[1]{\left(#1\right)}
\newcommand{\set}[1]{\left\{#1\right\}}

\newcommand{\Th}{\mathcal{T}_h}
\newcommand{\Thact}{\mathcal{T}_h^{act}}
\newcommand{\Thcut}{\mathcal{T}_h^{cut}}
\newcommand{\Thint}{\mathcal{T}_h^{int}}

\newcommand{\Fh}{\mathcal{F}_h}
\newcommand{\Fint}{\mathcal{F}^{int}_h}

\newcommand{\Pl}{\mathbb{P}_1}
\newcommand{\R}{\mathbb{R}}

\newcommand{\II}{\mathrm{I}}

\newcommand{\RR}{\mathrm{R}}
\newcommand{\hioQ}{H^{1,0}(Q)}
\newcommand{\hsioQ}{H_s^{1,0}(Q)}
\newcommand{\hnegioQ}{H^{-1,0}(Q)}
\newcommand{\ciQ}{C^1(\overline{Q})}
\newcommand{\coiQ}{C_0^1(Q)}

\newcommand{\afrac}{a^{\frac{1}{2}}}

\newcommand{\hfrac}{h^{\frac{1}{2}}}
\newcommand{\hnegfrac}{h^{-\frac{1}{2}}}
\newcommand{\al}{\alpha}
\newcommand{\be}{\beta}
\newcommand{\de}{\delta}
\newcommand{\defrac}{\delta^{\frac{1}{2}}}
\newcommand{\denegfrac}{\delta^{-\frac{1}{2}}}

\newcommand{\ep}{\varepsilon}
\newcommand{\ga}{\gamma}
\newcommand{\gafrac}{\gamma^{\frac{1}{2}}}
\newcommand{\ganegfrac}{\gamma^{-\frac{1}{2}}}

\newcommand{\na}{\nabla}

\newcommand{\Om}{\Omega}
\newcommand{\wOm}{\widetilde{\Omega}}
\newcommand{\wQ}{\widetilde{Q}}

\newcommand{\Qact}{Q^{act}_h}
\newcommand{\Fgp}{\mathcal{F}_h^{GP}}
\newcommand{\pa}{\partial}

\newcommand{\Si}{\Sigma}

\newcommand{\rd}{\,\mathrm{d}}

\newtheorem{THeorem}{Theorem}

    \newtheorem{LEmma}[THeorem]{Lemma}
	
    \newtheorem{REmark}{Remark}[section]
    \newtheorem{ASsumption}{Assumption}
    \newtheorem{EXample}{Example}[section]

\begin{document}
\begin{frontmatter}

\title{A Stabilized Unfitted Space-time Finite Element Method for Parabolic Problems on Moving Domains}
\author[nju]{Ruizhi Wang\fnref{myfootnote}}
\ead{rzwang@smail.nju.edu.cn}

\author[nju]{Weibing Deng\corref{mycorrespondingauthor}\fnref{myfootnote1}}
\ead{wbdeng@nju.edu.cn}
\cortext[mycorrespondingauthor]{Corresponding author}
\fntext[myfootnote1]{The work of this author was
supported by the National Key R\&D Program of China (2024YFA1012600), and by the NSF of China grant 12171237.}
\address[nju]{School of Mathematics,
Nanjing University, Nanjing 210093, People's Republic of China}

\begin{abstract}
This paper presents a space-time finite element method (FEM) based on an unfitted mesh for solving parabolic problems on moving domains. Unlike other unfitted space-time finite element approaches that commonly employ the discontinuous Galerkin (DG) method for time-stepping, the proposed method employs a fully coupled space-time discretization. To stabilize the time-advection term, the streamline upwind Petrov–Galerkin (SUPG) scheme is applied in the temporal direction. A ghost penalty stabilization term is further incorporated to mitigate the small cut issue, thereby ensuring the well-conditioning of the stiffness matrix. Moreover, an a priori error estimate is derived in a discrete energy norm, which achieves an optimal convergence rate with respect to the mesh size. In particular, a space-time Poincaré–Friedrichs inequality is established to support the condition number analysis. Several numerical examples are provided to validate the theoretical findings.
\end{abstract}

\end{frontmatter}


\setcounter{equation}{0}\qquad 
\numberwithin{equation}{section}
\setcounter{THeorem}{0}
\numberwithin{THeorem}{section}

\section{Introduction}\label{Sec1}
Parabolic problems on moving domains arise in various fields such as physics, chemistry, engineering, and biomedical applications. Typical examples include the Stefan problem involving melting and solidification processes \cite{stefanbook}, mass transport in multiphase flows \cite{twophase}, and fluid–structure interactions \cite{fluid}. The classical FEM for solving parabolic problems is the method of lines, which proceeds by first discretizing in space and then in time \cite{parabolic}. Although this approach is efficient for stationary domains, it becomes difficult to apply to moving domain problems when using conventional time-marching strategies such as the backward Euler method. The main challenge lies in the insufficient transfer of information from the previous time level to the current one, due to domain motion. Another widely used technique is the arbitrary Lagrangian–Eulerian (ALE) method \cite{ALE}, which maps the moving domain to a fixed reference domain where meshing is performed. However, the ALE method tends to become inefficient in cases involving large deformations or topological changes of the domain.

Instead of treating space and time separately, the space-time FEMs consider time as an additional spatial dimension, thereby raising the problem by one dimension. For parabolic problems, most space-time methods employ finite element spaces that are continuous in space but discontinuous in time, solving the discrete problem within each time slab (see for example \cite{stdg1978,BabuskaSTDG,ErikSTDG}).  Methods that are discontinuous in both space and time have also been explored, as seen in \cite{stDGDG,stDGJJW}. 

Recently, growing attention has been directed toward space-time FEMs that are continuous in both space and time \cite{SteinbachST,Moore2016,STDG,Moore2018,Langer2019,fitST}. These methods discretize the entire space-time domain using unstructured meshes and apply standard finite element techniques. A key advantage of this approach is its suitability for parallel computation over the full space-time domain, as well as its potential for space-time adaptive refinement. For instance, Steinbach et al. \cite{SteinbachST} introduced an unstructured conforming space-time FEM for general parabolic equations.
Moore et al. \cite{Moore2016} proposed a space-time isogeometric analysis using time-upwind test functions, applicable to both fixed and moving domains, which was later extended to multipatch isogeometric analysis in \cite{STDG}.
Localized stabilization strategies were further developed in \cite{Moore2018,Langer2019}. Additionally, a fitted space-time FEM was presented in \cite{fitST} for advection-diffusion problems with moving interfaces. For a comprehensive overview of space-time FEMs for parabolic problems, we refer to \cite{SToverview}.
However, when dealing with complex moving domain problems, the space-time methods mentioned above require the generation of a fitted and regular space-time mesh. This task becomes particularly challenging in three-dimensional spatial settings, where it entails constructing four-dimensional simplex meshes \cite{fitmesh}.

To circumvent the challenges of generating high-quality fitted meshes, unfitted methods have attracted considerable attention in recent decades. Notable examples include the CutFEM \cite{cutFEM}, the finite cell method \cite{FCM}, the shift boundary method \cite{SBM}, the aggregated unfitted FEM \cite{BADIA}, the extended FEM (XFEM) \cite{xfem}, the generalized FEM \cite{gfem} and the immersed interface method \cite{ifem}. These approaches eliminate the need for body-fitted meshes, allowing the use of simple background grids—such as Cartesian meshes—thereby significantly reducing computational cost. 
However, unfitted methods often face the so-called small cut problem. When the background mesh is arbitrarily intersected by the physical domain, certain elements may contain only very small cut portions, leading to a severely ill-conditioned stiffness matrix that can cause numerical instability \cite{cdnum}. To address this issue, two main types of stabilization techniques have been developed. The first approach introduces additional stabilization terms into the bilinear form, as seen in the ghost penalty method \cite{BURMAN10,BURMAN12} and the finite cell method \cite{FCM}. The second strategy aggregates small cut elements into larger patches to improve matrix conditioning, a technique employed in aggregated unfitted FEMs and related works (see \cite{Johansson,BADIA,HUANG,Chen2021}).

Several unfitted FEMs have been developed for solving time-dependent problems involving moving domains or moving interfaces. In \cite{EulerFEM,EulerStokes}, a time-stepping scheme is introduced for time-dependent moving domain problems, which requires a discrete extension of the solution from the previous time step to the domain at the next time step.
This idea has also been applied in \cite{CLAUS} to solve two-phase flow problems with moving interfaces. In \cite{cutStefan}, a CutFEM for the one-phase Stefan--Signorini problem with applications in laser manufacturing is proposed. In \cite{MA}, an interface-tracking algorithm combined with high-order unfitted characteristic finite element methods is developed for moving interface problems governed by the Oseen equations. Furthermore, several immersed finite element methods incorporating time-stepping strategies have been proposed in \cite{IFEM1,IFEM2,IFEM3} for solving parabolic moving interface problems. 
Meanwhile, a space-time FEM based on discontinuous Galerkin (DG) time-stepping has been proposed in \cite{Preuss,XFEMdg,XFEMdg3d,BADIA2023} to handle moving domains and interfaces. This approach has also been adapted to moving surface problems \cite{surface} and coupled bulk-surface problems \cite{bulkSurface}. More recently, the method has been extended to a continuous-in-time space-time formulation \cite{Heimann2021,Heimann2023}.
However, all the aforementioned methods are sequential in time and require sufficiently small time steps to maintain accuracy when the domain undergoes high oscillations or large deformations over time.

In this work, we propose a new unfitted space-time FEM that employs an unstructured mesh over the space-time domain along with piecewise linear finite elements. The use of unfitted meshes significantly simplifies mesh generation and reduces computational cost compared to fitted space-time approaches. A key advantage of our method over unfitted discontinuous Galerkin (DG) time-stepping schemes lies in its flexibility in space-time discretization. This flexibility enables local mesh refinement in regions where the domain undergoes rapid motion or large deformations, thereby facilitating fully adaptive resolution in both space and time while reducing the overall degrees of freedom. Furthermore, the resulting discrete system can be solved in a fully parallel manner via domain decomposition and efficient space-time preconditioning techniques, as exemplified in references \cite{Moore2016,AMG}. 
To mitigate the small cut issue, ghost penalty stabilization is employed to ensure a well-conditioned stiffness matrix. A streamline upwind Petrov–Galerkin (SUPG) scheme, inspired by \cite{Moore2016,Langer2019}, is applied to stabilize the time-advection term.

Numerical experiments in one and two spatial dimensions show that the proposed method attains the optimal convergence rate in the $H^{1,0}$ norm, along with an optimal growth rate for the condition number of the stiffness matrix. The robustness of the method is verified through a small cut test, confirming its stability irrespective of how the physical domain intersects the background mesh. A boundary layer problem is also presented to highlight the effectiveness of the SUPG scheme.

We emphasize that the primary theoretical challenge lies in the condition number estimate, which requires bounding the $L^2$ norm by a discrete energy norm. For stationary domains, this estimate is typically derived by first applying the Poincaré–Friedrichs inequality in space and then integrating in time. In the case of moving domains, however, the constant in the Poincaré–Friedrichs inequality becomes time-dependent. To integrate this inequality over the time interval, the explicit dependence of the constant on the domain variation must be established. Under mild geometric assumptions and with the aid of several auxiliary results from \cite{PF1d, maggi2005}, we have successfully proved such an inequality. 

The remainder of this paper is organized as follows. Section \ref{Sec2} introduces the model problem posed on a  moving domain and presents its variational formulation. Section \ref{Sec3} details the proposed unfitted space-time FEM. The corresponding energy error analysis is provided in Section \ref{Sec4}.  In Section \ref{Sec5}, we initially establish a space-time Poincaré–Friedrichs inequality under mild geometric assumptions, followed by deriving an estimate for the condition number. Numerical experiments to validate our theoretical findings are presented in Section \ref{Sec6}. Finally, Section \ref{Sec7} concludes the paper.

Throughout the paper, the letter $C(C_0,C_1,\cdots)$ represents a general constant that is independent of the mesh size and the position of the boundary relative to the meshes. We also use the shorthand notation $a \lesssim b$ and $b \gtrsim a$ for the inequality $a \leq Cb$ and $b \geq Ca$. 

\section{Model problem}\label{Sec2}

We begin by introducing notation for the moving domain. Let $\Omega(t) \subset \mathbb{R}^d$, with $d = 1, 2, 3$, be an open bounded domain whose boundary $\partial \Omega(t)$ is Lipschitz (for $d = 2, 3$). Here, $t$ indicates the time dependence of the domain. For simplicity, we denote $\Omega(t)$ as $\Omega_t$ and $\partial \Omega(t)$ as $\partial \Omega_t$ when appropriate.
Moreover, we assume that the motion of the domain can be determined in advance based on prior knowledge--for instance, through a level set function $\phi(x,t)$ such that $\Om(t) = \{ x \in \mathbb{R}^{d}: \phi(x,t) < 0 \}$.

Let $T > 0$ and $I = (0, T)$ be the time interval. The associated space-time domain is defined as
$Q := \bigcup_{t \in I} \Omega(t) \times \{t\}$.
The boundary of $Q$ is composed of three parts: the lateral (moving) boundary $\Sigma_s := \bigcup_{t \in I} \partial \Omega(t) \times \{t\}$, the initial boundary $\Sigma_0 := \Omega(0) \times \{0\}$, and the final boundary $\Sigma_T := \Omega(T) \times \{T\}$.
We assume that $\partial Q$ is Lipschitz continuous and piecewise $C^2$-smooth. Moreover, we denote the space gradient operator by $\na_x$ and the space-time gradient operator by $\na$, i.e., $ 
\na_x = \bk{ \frac{\pa}{\pa x_1}, \frac{\pa}{\pa x_2}, \cdots, \frac{\pa}{\pa x_d} }^T, 
\na = \bk{ \frac{\pa}{\pa x_1}, \frac{\pa}{\pa x_2}, \cdots, \frac{\pa}{\pa x_d}, \frac{\pa}{\pa t} }^T
$. 

In this paper, we consider the following parabolic problem on a moving domain:
\begin{equation}\label{P}
	\begin{cases}
		\pa_t u - \na_x \cdot ( a(x,t) \na_x u ) = f & \text{in} ~Q, \\
		u = g & \text{on} ~\Sigma_s, \\
		u(x,0) = u_0 & \text{on} ~\Sigma_0,
	\end{cases}
\end{equation}
where $f \in L^2(Q)$, $g \in H^{\frac{1}{2}}(\Si_s)$, $u_0 \in H^1(\Om_0)$, with $u_0$ and $g$ satisfying the compatibility conditions. The time-dependent diffusion coefficient $a(x,t) \in C^0(\overline{Q})$ is uniformly bounded from below and above by some positive constants, namely,
\begin{equation*}
	0 < a_{min} \leq a(x,t) \leq a_{max},  \quad \forall (x,t) \in Q,
\end{equation*}
and it is also assumed that $\na_x a(x,t) \in L^{\infty}(Q)$. The Sobolev spaces used here will be defined subsequently.

We work under the following mild geometric assumptions:
\begin{ASsumption}\label{assump1}
    There exist positive constants $C_{\Om}, C_{\pa \Om}$ independent of $t$ such that $\snorm{\Om(t)} \leq C_{\Om}$ and $\snorm{\pa\Om(t)} \leq C_{\pa \Om}$ for any $t \in [0,T]$.
\end{ASsumption}
\begin{ASsumption}\label{assump2}
    There exists a finite collection of open or closed bounded subsets $\set{O_i}_{i=1}^N$ in $\R^{d+1}$ such that 
    \begin{itemize}
        \item Each $O_i$ covers a neighborhood of $\Si_s$ and is contained in a time strip $\al_i < t < \be_i$ (or $\al_i \leq t \leq \be_i$), with $\al_i \geq 0$.

\item For each $i = 1, \dots, N$, there exists a $C^2$-diffeomorphism $h_i: \overline{O_i} \to \overline{B(\alpha_i, \beta_i)}$, where
\begin{equation*}
	B(\alpha_i, \beta_i) = \big\{ (\xi, \tau) \in \mathbb{R}^{d+1} : -1 < \xi_k < 1,\ k = 1, \dots, d,\ \alpha_i < \tau < \beta_i \big\},
\end{equation*}
and
\[
h_i(x,t) = \big( \xi_1(x,t), \dots, \xi_d(x,t), \tau(x,t) \big).
\]
We assume that $\tau(x,t) = t$, and that $h_i$ maps:  $\overline{Q \cap O_i}$ bijectively onto $\overline{B(\alpha_i, \beta_i)} \cap \{ \xi_d \geq 0 \}$,
	and  $\overline{\Sigma_s \cap O_i}$ bijectively onto $\overline{B(\alpha_i, \beta_i)} \cap \{ \xi_d = 0 \}$.
        
\end{itemize}
    
\end{ASsumption}

\begin{REmark}
   Together, these assumptions enforce a sufficiently smooth evolution of the moving domain. Specifically, Assumption \ref{assump1} is essential for establishing the space-time Poincaré–Friedrichs inequality, whereas Assumption \ref{assump2} guarantees the validity of the trace operator $\gamma_s: H^{1,0}(Q) \rightarrow L^2(\Sigma_s)$. The latter was originally introduced in \cite{Lions} for analyzing parabolic systems in non-cylindrical domains.
\end{REmark}
\begin{REmark}
    Since $\Si_s$ is piecewise $C^2$ smooth, Assumption \ref{assump2} implies that its non-smooth part is contained in the boundary of a closed set $O_i$, forming a lower-dimensional subset of $\Si_s$.
\end{REmark}

For a real number $r\geq 0$ and a measurable set $U \subset \R^{d+1}$, let $H^r(U)$ denote the standard Sobolev space equipped with the usual norm $\norm{\cdot}_{r,U}$ and semi-norm $\snorm{\cdot}_{r,U}$. By convection, for $r=0$,  we write $\bk{\cdot,\cdot}_U$ and $\norm{\cdot}_U$ for the $L^2(U)$ inner product and norm, omitting the subscript $r$.

We now introduce anisotropic Sobolev spaces following \cite{wuZhuoqun}. Let $\al = (\al_1,\cdots,\al_d)$ be a  multi-index with $|\al| = \sum_{i=1}^d \al_i$. For any $m,n \in \mathbb{N}$ and any measurable set $U \subset \R^{d+1}$, define the Hilbert space
\begin{equation*}
    H^{m,n}(U) = \set{ u \in L^2(U): \pa_x^{\al}u, \pa_t^k u \in L^2(U),~ \forall\, |\al|\leq m, ~0\leq k \leq n},
\end{equation*}
equipped with the norm
\begin{equation*}
    \normx{u}_{H^{m,n}(U)}^2 = \sum_{|\al|\leq m} \normx{\pa_x^{\al}u}_{U}^2  + \sum_{0\leq k \leq n} \normx{\pa_t^k u}_{U}^2.
\end{equation*}
Under Assumption \ref{assump2}, it is shown in \cite{Lions} that the space $\hioQ$ can be characterized as the closure of $\ciQ$ w.r.t. the norm $\normx{\cdot}_{\hioQ}$, i.e.,
\begin{equation*}
    \hioQ = \overline{\ciQ}^{\normx{\cdot}_{\hioQ}}.
\end{equation*}
Similarly, define $\hsioQ$ as the closure of $\coiQ$ under $\normx{\cdot}_{\hioQ}$:
$\hsioQ = \overline{\coiQ}^{\normx{\cdot}_{\hioQ}}$,
and denote by $\hnegioQ$ the dual space of $\hsioQ$. 

To establish the weak formulation, we introduce additional function spaces, whose definitions and properties can be found in \cite{Lions}. 
For $u \in \hioQ$, the distributional time derivative $\pa_t u$ is defined as the linear functional 
\begin{equation*}
    \pd{\pa_t u, \phi} := -\int_Q u \pa_t \phi \rd x \rd t, \quad \forall \phi \in \coiQ,
\end{equation*}
where $\pd{\cdot,\cdot}$ denotes the duality pairing. 
Under Assumption \ref{assump2}, there exists a continuous linear trace operator $\ga_s: \hioQ \rightarrow L^2(\Si_s)$ and the space $\hsioQ$ can be characterized as 
\begin{equation}\label{Hs10}
    \hsioQ = \set{ u \in \hioQ: \ga_s u = 0 ~\text{on}~ \Si_s}.
\end{equation}
Here, $\ga_s u$ is also written as $u|_{\Si_s}$. 
Furthermore, we define the space
\begin{equation*}
    W = \set{ u \in \hioQ: \pa_t u \in \hnegioQ },
\end{equation*}
with the norm
\begin{equation*}
    \normx{u}_{W}^2 = \normx{u}_{\hioQ}^2 + \normx{\pa_t u}_{\hnegioQ}^2.
\end{equation*}

The weak solution of \eqref{P} is to find $u \in W$, $u|_{\Si_s} = g$, $u|_{\Si_0} = u_0$, such that
\begin{equation}\label{vf1}
    A(u,v) := \pd{ \pa_t u, v} + \bk{a \na_x u, \na_x v}_Q = \bk{f,v}_Q, \quad \forall v \in \hsioQ.
\end{equation}
The well-posedness of the variational formulation \eqref{vf1} follows by homogenizing the boundary conditions \cite{LionsV1} and applying \cite[Lemma 3.1]{Lions}.

\begin{THeorem}
    Suppose Assumption \ref{assump2} holds. Then for every $f \in L^2(Q)$, $g \in H^{\frac{1}{2}}(\Si_s)$, $u_0 \in H^1(\Om_0)$, there exists a unique solution $u \in W$ to problem \eqref{vf1}.
\end{THeorem}

\section{Unfitted space-time FEM}\label{Sec3}
In this section, we discretize the variational formulation \eqref{vf1} using an unfitted space-time FEM. We begin by introducing some necessary notations. Let $\wOm \subset \R^d$ be a polygonal background domain that contains the physical domain $\Om(t)$ for all $t \in I$, i.e., $\Om(t) \subset \wOm$, and define the background space-time domain as $\wQ = \wOm \times I$. 
Let $\Th$ be a family of conforming triangulations of $\wQ$, which are not required to align with the moving boundary $\Si_s$.
Each element $K \in \Th$ is considered as closed. Define the local mesh size $h_K = \text{diam}(K)$ and the global mesh size $h = \max_{K \in \Th} h_K$. 
We assume that $\Th$ is shape regular and quasi-uniform, in the sense that there exist two constants $\mu, \nu > 0$ such that each element $K \in \Th$ contains a ball of diameter $\rho_K$ with $h_K / \rho_K \leq \mu$ and
\begin{equation} \label{quasiUniform}
    h/h_K \leq \nu, \quad \forall K \in \Th.
\end{equation}
We introduce the so-called active mesh 
\begin{equation*}
    \Thact := \{ K \in \Th: K \cap Q \neq \emptyset \},   
\end{equation*}
and the corresponding active domain $\Qact := \bigcup_{K \in \Thact} K$. Similarly, the cut mesh is defined as
\begin{equation*}
    \Thcut := \{ K \in \Th: K \cap \Si_s \neq \emptyset \},  
\end{equation*}
and the interior mesh as $\Thint = \Thact \backslash \Thcut$.

We impose the following geometric assumption on the mesh regarding the connectivity between cut and interior elements. 
\begin{ASsumption}\label{assump3}
For every $K \in \Thcut$, there exists an element $K' \in \Thint$ and a path $\{K_1, K_2, \dots, K_N\} \subset \Thact$ such that $K_1 = K, K_N = K'$, and consecutive elements $K_i$ and $K_{i+1}$ share a common facet for $i = 1, \dots, N-1$. The path length $N$ is uniformly bounded by a constant independent of both the mesh size $h$ and the position of $\Si_s$ relative to $\Th$.
\end{ASsumption}
We note that this assumption is satisfied provided the mesh is sufficiently fine to properly resolve the geometry of the space-time domain $Q$.

Next, we introduce some notations related to mesh faces. Let $\Fh$ denote the set of all element faces in $\Th$, and define the following subsets:
\begin{align*}
    \Fint &:= \{ F \in \Fh : F \nsubseteq \pa \wQ \},\\
    \Fgp &:= \{ F = \pa K_1 \cap \pa K_2 \in \Fh : F \cap Q \neq \emptyset, 
    ~K_1 \in \Thcut ~\text{or}~ K_2 \in \Thcut \}.
\end{align*}
Each face $F \in \Fh$ is assigned a unit normal vector $n_F$. For a boundary face $F\subset\pa \wQ$, we take $n_F$ to be the unit outward normal to $\pa \wQ$. For an interior face $F \in \Fint$, the jump of a function $w$ defined on $\bigcup_{K \in \Th} K$ is given by
\begin{equation*}
    [w]|_F = w_F^+ - w_F^-,
\end{equation*}
where $w_F^{\pm} = \lim_{t \rightarrow 0^+} w(x \pm tn_F)$. Let $(n_x, n_t)$ denote the unit outward normal vector to the moving boundary $\Sigma_s$, where $n_x$ and $n_t$ represent the spatial and temporal components, respectively.

We denote by $(\cdot, \cdot)_{\omega}$ the $L^2$ inner product on any measurable subset $\omega \subset \mathbb{R}^{d+1}$, and by $\langle \cdot, \cdot \rangle_s$ the inner product on any measurable set $s \subset \mathbb{R}^{d+1}$ of codimension one.

For any subcollection $\mathcal{T} \subset \mathcal{T}_h$, any measurable set $U \subset \mathbb{R}^{d+1}$ and any norm $\norm{\cdot}$, we define
\[
(\cdot, \cdot)_{\mathcal{T} \cap U} = \sum_{K \in \mathcal{T}} (\cdot, \cdot)_{K \cap U}, \quad
\| \cdot \|_{\mathcal{T} \cap U}^2 = \sum_{K \in \mathcal{T}} \| \cdot \|_{K \cap U}^2.\]
Analogous definitions apply when the subscripts $\mathcal{T} \cap U$ and $K \cap U$ are substituted by $\mathcal{T}$ and $K$, respectively.

Let $\mathcal{L}_m$ denote the $m$-dimensional Lebesgue measure. For any measurable set $U \subset \mathbb{R}^{d+1}$, we write $|U| = \mathcal{L}_{d+1}(U)$; if $U$ is a $d$-dimensional manifold, we write $|U| = \mathcal{L}_d(U)$.

For each element $K \in \Th$, let $\Pl(K)$ denote the space of linear polynomials on $K$. The finite element space defined on the active mesh is  given by
\begin{equation*}
    V_h = \{ v_h \in H^1(\Qact): v_h|_K \in \Pl(K), ~\forall K \in \Th \}.
\end{equation*}
The space-time FEM then reads: find $u_h \in V_h$ such that for all $v_h \in V_h$ 
\begin{equation} \label{FEM}
    A_h(u_h,v_h) := a_h(u_h,v_h) + g_h(u_h,v_h) = l_h(v_h),
\end{equation}
where the bilinear form $a_h$ and the linear form $l_h$ are defined as
\begin{equation}\label{ah}
    \begin{aligned}
    a_h(u,v) :=& (\pa_t u, v)_Q + (a\na_x u,\na_x v)_Q - \pd{a\na_x u \cdot n_x, v}_{\Si_s} - \pd{a\na_x v \cdot n_x, u}_{\Si_s}  \\
    & + \pd{\ga h^{-1}u, v}_{\Si_s} + \pd{u,v}_{\Si_0} + \bk{\de h^2 ( \pa_t u-\na_x \cdot (a \na_x u)), \pa_t v}_{\Thact \cap Q},
    \end{aligned}
\end{equation}
\begin{equation*}
    l_h(v) := \bk{f, v}_Q - \pd{a\na_x v \cdot n_x, g}_{\Si_s} + \pd{\ga h^{-1}g, v}_{\Si_s} + \pd{u_0, v}_{\Si_0} + \bk{\de h^2 f, \pa_t v}_{\Thact \cap Q},
\end{equation*}
and the ghost penalty term $g_h$ is defined as
\begin{equation*}
    g_h(u,v) = \sum_{F \in \Fgp} \pd{\ga_1 h 
    [\pa_{n_F} u], [\pa_{n_F} v] }_F.
\end{equation*}
Here, $\ga, \ga_1$ and $\de$ are mesh-independent positive constants to be specified later. 

\begin{REmark}
   The technique employed above for the weak enforcement of Dirichlet boundary conditions is based on the symmetric Nitsche’s penalty method \cite{Nitsche77}. It is also widely used in the context of the interior penalty discontinuous Galerkin (IPDG) method \cite{arnold82}. 
\end{REmark}

\begin{REmark}
    The final term in the bilinear form \eqref{ah} originates from the SUPG scheme \cite{supg}, which is employed to stabilize the norm of the time derivative in the numerical solution. This approach can be interpreted as replacing the original test space $V_h$ with the enriched space:
\[
\left\{ v_h + \delta h^2 \, \partial_t v_h : v_h \in V_h \right\}.
\]
The idea of incorporating time-upwind test functions in space-time methods has also been adopted in \cite{Moore2016,Moore2018}. When the diffusion coefficient $a(x, t)$ is constant, the last term in \eqref{ah} simplifies to $\bk{\de h^2 \pa_t u_h, \pa_t v_h}$, which corresponds to the weak form of the operator $-\delta h^2 \, \partial_{tt}$. Thus, this technique effectively introduces numerical diffusion along the temporal direction.
\end{REmark}

\begin{REmark}
    The SUPG term is scaled with $h^2$ in the space-time mesh since we have assumed the shape regularity and quasi-uniformity of $\Th$.
    This mesh parameter can be localized to $h_K^2$ if $\Th$ is shape regular but not quasi-uniform. Analogous parameter selection is also valid for the terms $\pd{\ga h^{-1}u,v}_{\Si_s}$ and $g_h(u,v)$, where the mesh parameter $h$ can be replaced by $h_K$ and $h_F:= \text{diam}(F)$, respectively. Moreover, the penalty parameters $\ga$ and $\ga_1$ can also be localized.
\end{REmark}

\begin{REmark}
The ghost penalty stabilization was first introduced in \cite{BURMAN10} for the Poisson equation to improve the conditioning of the stiffness matrix resulting from small cut elements. The specific form employed in this work is referred to as the face-based stabilization \cite{BURMAN12,cutFEM}. Two other commonly used variants are the patch-based \cite{BURMAN10,EulerFEM} and element-based \cite{Preuss,EulerFEM} stabilizations. Alternative techniques for designing such stabilization terms have also been proposed, including the approach in \cite{burman22}.
\end{REmark}

\section{A priori error analysis}\label{Sec4}
This section presents an a priori error estimate for the proposed method. Throughout the remainder of the paper, we assume that Assumptions \ref{assump1}--\ref{assump3} hold, unless stated otherwise.
To this end, we define the energy space as $V = \set{ v \in H^1(Q): v|_{K \cap Q} \in H^{2,1}(K \cap Q), ~\forall K \in \Th }$.
To facilitate the error estimates, we introduce the following norms and semi-norms. For any $v \in V_h+V$, define 
\begin{align*}
    \norme{v}^2 &= \normx{\afrac \na_x v}_Q^2 + \normxss{\gafrac \hnegfrac v}^2 + \normxsz{v}^2 + \normxst{v}^2 + \normx{\defrac h \pa_t v}_Q^2,\\
    \normes{v}^2 &= \norme{v}^2 + \normxss{\ganegfrac \hfrac a\na_x v \cdot n_x}^2 + \normx{\denegfrac h^{-1}v}_Q^2 + \normx{\defrac h\na_x \cdot (a\na_x v)}_{\Thact \cap Q}^2,
\end{align*}
and for any $v \in V_h$, define 
\begin{equation*}
    \snormgh{v}^2 = g_h(v,v),\quad
    \normeh{v}^2 = \norme{v}^2 + \snormgh{v}^2,
\quad
    \normehs{v}^2 = \normes{v}^2 + \snormgh{v}^2.
\end{equation*}

\subsection{Some auxiliary results}

We require the following local trace inequality for cut elements, as stated in \cite[Lemma 1]{Ginfsup}.
\begin{LEmma} \label{cutlem}
    Assume that $\pa Q$ is Lipschitz. Then, for any $K \in \Thcut$ and for $h < h_0$, where $h_0$ is an arbitrary but fixed constant, the following estimates hold:
    \begin{align}
        \norm{v}_{K \cap \pa Q} &\lesssim  h_K^{-1/2} \norm{v}_K +  h_K^{1/2} \norm{\na v}_K , \quad \forall v \in H^1(K),\label{cutTr}\\
        \norm{v_h}_{K \cap \pa Q} &\leq C_{inv} h_K^{-1/2} \norm{v_h}_K, \quad \forall v_h \in \Pl(K),\label{cutTrpoly}
    \end{align}
    where $C_{inv}$ is a positive constant independent of both the mesh size and the position of the boundary $\Si_s$ relative to the mesh.
\end{LEmma}
In applications of these trace inequalities, we may take $h_0 = \operatorname{diam}(\widetilde{Q})$. We also note that variants of the above estimates have been established under different assumptions in \cite{hh02,wx10,Chen2021}.

The following lemma shows that ghost penalty $g_h$ extends both the $L^2$ norm and $H^1$ semi-norm from the physical domain $Q$ to the active domain $\Qact$, whose proof can be found in \cite{GURKAN2019466}.
\begin{LEmma}
    For any $v_h \in V_h$, it holds that
    \begin{align}
        \norm{v_h}_{\Qact}^2 &\leq C_{g,0} \bk{ \norm{v_h}_Q^2 + \snormgh{v_h}^2 },\label{ghL2}\\
        \norm{\na_x v_h}_{\Qact}^2 &\leq C_g  \bk{ \norm{\na_x v_h}_Q^2 + \snormgh{v_h}^2 },\label{ghH1x}\\
        \norm{\pa_t v_h}_{\Qact}^2 &  \leq \widetilde{C}_g  \bk{ \norm{\pa_t v_h}_Q^2 + \snormgh{v_h}^2 } ,\label{ghH1t}
    \end{align}
    where $C_{g,0}, ~C_g$ and $\widetilde{C}_g$  are positive constants that are independent of
    the mesh size and the position of the boundary $\Si_s$ relative to the mesh.
\end{LEmma}

\subsection{Stability analysis of the method}
With the above preliminaries, we now establish the continuity and coercivity of the bilinear form $A_h$ in the following two theorems. 
\begin{THeorem}\label{continuous}
    It holds that
    \begin{align}     
        a_h(u,v_h) &\lesssim \normes{u} \normeh{v_h}, \quad \forall\, u \in V_h+V, ~ v_h \in V_h,\label{ahct}\\
        A_h(u_h,v_h) &\lesssim \normehs{u_h} \normeh{v_h}, \quad \forall u_h \in V_h, ~ v_h \in V_h. \label{Ahct}
    \end{align}
\end{THeorem}
\begin{proof}
 We estimate each term in \eqref{ah} to establish \eqref{ahct}. For the first term, applying integration by parts in time followed by the Cauchy–Schwarz inequality yields
    \begin{align*}
        \bk{\pa_t u, v_h}_Q & = -\bk{u, \pa_t v_h}_Q - \pd{u,v_h}_{\Si_0} + \pd{u,v}_{\Si_T} + \pd{u n_t, v_h}_{\Si_s} \\
        & \leq\ \normx{\denegfrac h^{-1} u}_Q \normx{\defrac h \pa_t v_h}_Q + \normxsz{u} \normxsz{v_h} + \normxst{u} \normxst{v_h} \\ 
        & \quad + \ga^{-1}h \normxss{\gafrac \hnegfrac u} \normxss{\gafrac \hnegfrac v_h}.
    \end{align*}
    The fourth term in \eqref{ah} can be bounded using the Cauchy-Schwarz inequality, the trace inequality \eqref{cutTrpoly} and estimate \eqref{ghH1x}
    \begin{align*}
        \pd{a \na_x v_h \cdot n_x, u}_{\Si_s} 
        & \leq \normxss{\ganegfrac \hfrac a \na_x v_h \cdot n_x} \normxss{\gafrac \hnegfrac u} \\ 
        & \lesssim \ganegfrac a_{max} \normx{\na_x v_h}_{\Qact} \normxss{\gafrac \hnegfrac u} \\
        & \lesssim \bk{\normx{\na_x v_h}_Q + \snormgh{v_h} } \normxss{\gafrac \hnegfrac u} \\
        & \lesssim \bk{\normx{\afrac \na_x v_h}_Q + \snormgh{v_h} } \normxss{\gafrac \hnegfrac u}.
    \end{align*}
    The remaining terms in \eqref{ah} are bounded directly via the Cauchy-Schwarz inequality. To establish \eqref{Ahct}, it remains to estimate $g_h(u_h,v_h)$, which again follows from the Cauchy-Schwarz inequality:
    \begin{equation*}
        g_h(u_h,v_h) \leq \snormgh{u_h} \snormgh{v_h}.
    \end{equation*}
    This completes the proof.
\end{proof}

\begin{THeorem}\label{coercive}
    There exist two positive constants $\de_0, \ga_0 $ such that whenever $0<\de \leq \de_0$ and $\ga \geq \ga_0$, we have
    \begin{equation}\label{Ahcoer}
        A_h(v_h,v_h) \geq \frac{1}{2} \normeh{v_h}^2, \quad \forall v_h \in V_h.
    \end{equation}
\end{THeorem}
\begin{proof}
    Without loss of generality, we assume $h < 1$. 
    Using integration by parts in time, we obtain
    \begin{equation*}
        \bk{\pa_t v_h, v_h}_Q = \frac{1}{2} \int_Q  \pa_t(v_h^2) \rd x \rd t = \frac{1}{2} \int_{\pa Q} v_h^2 n_t \rd s 
        = -\frac{1}{2} \normxsz{v_h}^2 + \frac{1}{2} \normxst{v_h}^2 + \frac{1}{2} \pd{v_h n_t, v_h}_{\Si_s}.
    \end{equation*}
 Substituting this into the definition of $A_h$, we derive
  \begin{equation}\label{threeterm}
    	\begin{split}
        A_h(v_h,v_h)
        & = \normx{\afrac \na_x v_h}_Q^2 + \snormgh{v_h}^2 + \normxss{\gafrac \hnegfrac v_h}^2 + \frac{1}{2}\normxsz{v_h}^2 + \frac{1}{2}\normxst{v_h}^2 + \normx{\defrac h \pa_t v_h}_Q^2  \\
        & \quad + \frac{1}{2} \pd{v_h n_t,v_h}_{\Si_s} - 2\pd{a\na_x v_h\cdot n_x, v_h}_{\Si_s} - \de h^2 \bk{\na_x a \cdot \na_x v_h, \pa_t v_h}_{\Thact \cap Q}. 
        \end{split}
    \end{equation}
    We now estimate the last three terms in \eqref{threeterm}.  For the first term, we have
    \begin{equation} \label{term1}
        \frac{1}{2} \pd{v_h n_t,v_h}_{\Si_s} \leq \frac{1}{2} \normxss{v_h}^2 = \frac{1}{2} \ga^{-1}h \normxss{\gafrac \hnegfrac v_h}^2 \leq \frac{1}{2} \ga^{-1} \normxss{\gafrac \hnegfrac v_h}^2.
    \end{equation}
    To bound the second term, we apply the Cauchy-Schwarz inequality, the $\varepsilon$-inequality $2ab \leq \varepsilon a^2 + \varepsilon^{-1}b^2$, the trace inequality \eqref{cutTrpoly}, the quasi-uniformity \eqref{quasiUniform} and the estimate~\eqref{ghH1x}
    \begin{equation} \label{term2}
    \begin{aligned}
        2\pd{a\na_x v_h\cdot n_x, v_h}_{\Si_s} 
        & \leq 2 \ganegfrac a_{max} \normxss{\hfrac \na_x v_h} \normxss{\gafrac \hnegfrac v_h} \\
        & \leq \varepsilon a_{max}^2 \normxss{\hfrac \na_x v_h}^2 + \varepsilon^{-1} \ga^{-1} \normxss{\gafrac \hnegfrac v_h}^2 \\
        & \leq C_1 \varepsilon \normx{\na_x v_h}_{\Qact}^2 + \varepsilon^{-1} \ga^{-1} \normxss{\gafrac \hnegfrac v_h}^2 \\
        & \leq C_1C_g \varepsilon \bk{ \normx{\na_x v_h}_Q^2 + \snormgh{v_h}^2 } + \varepsilon^{-1} \ga^{-1} \normxss{\gafrac \hnegfrac v_h}^2 \\
        & \leq C_2C_g \varepsilon \bk{ \normx{\afrac \na_x v_h}_Q^2 + \snormgh{v_h}^2 } + \varepsilon^{-1} \ga^{-1} \normxss{\gafrac \hnegfrac v_h}^2, 
    \end{aligned}
    \end{equation}
    where $C_1 = a_{max}^2 C_{inv}^2 \nu $ and $C_2 = C_1 \max \set{a_{min}^{-1}, 1}$. 
    The third term is bounded using the Cauchy-Schwarz inequality and Young's inequality
    \begin{equation} \label{term3}
    \begin{aligned}
        \de h^2 \bk{\na_x a \cdot \na_x v_h, \pa_t v_h}_{\Thact \cap Q} 
        & \leq \normx{\defrac h \na_x a \cdot \na_x v_h}_Q \normx{\defrac h \pa_t v_h}_Q \\
        & \leq \frac{1}{2} \de \normx{ h \na_x a \cdot \na_x v_h}_Q^2 + \frac{1}{2}  \normx{\defrac h \pa_t v_h}_Q^2 \\
        & \leq C_3 \de \normx{\na_x v_h}_{Q}^2 + \frac{1}{2}\normx{\defrac h \pa_t v_h}_Q^2  \\
        & \leq C_4 \de \normx{\afrac \na_x v_h}_{Q}^2 + \frac{1}{2}\normx{\defrac h \pa_t v_h}_Q^2, 
    \end{aligned}
    \end{equation}
    where $C_3 = \normx{\na_x a}_{L^{\infty}(Q)}^2/2$ and $C_4 = C_3 a_{min}^{-1}$. 
    Combining \eqref{threeterm}--\eqref{term3}, we obtain
    \begin{align*}
        A_h(v_h,v_h) &\geq \bk{1-C_2C_g\varepsilon - C_4 \de} \bk{\normx{\afrac \na_x v_h}_Q^2 + \snormgh{v_h}^2} + \bk{1-\varepsilon^{-1}\ga^{-1}-\frac{1}{2} \ga^{-1}} \normxss{\gafrac \hnegfrac v_h}^2 \\
        & \quad + \frac{1}{2} \bk{ \normx{\defrac h \pa_t v_h}_Q^2 + \normxsz{v_h}^2 + \normxst{v_h}^2 } \\
        & \geq \frac{1}{2} \normeh{v_h}^2,
    \end{align*}
    where the last inequality follows by choosing $\varepsilon = \bk{4C_2C_g}^{-1}$, $\de \leq \bk{4C_4}^{-1}$, $\ga \geq 8C_2C_g+1$.
\end{proof}
\begin{REmark}
    If $a(x,t)$ is a constant, then the last term in \eqref{threeterm} vanishes. Consequently, the condition $\de \leq \de_0$ can be omitted in Theorem \ref{coercive}.
\end{REmark}
By virtue of norm equivalence in finite-dimensional spaces, along with Theorem \ref{continuous}, Theorem \ref{coercive}, and the Lax-Milgram theorem, we deduce the existence and uniqueness of a solution $ u_h \in V_h $ to the finite element method \eqref{FEM}.

\subsection{Energy error estimates}
In this subsection, we present the derivation of the error estimate in the energy norm $\norme{\cdot}$ for the proposed method. Let $u \in V$ and $u_h \in V_h$ denote the solutions of problems \eqref{P} and \eqref{FEM}, respectively. One can readily show that
    \begin{equation}\label{GaOrth}
        a_h(u-u_h,v_h) = g_h(u_h,v_h), \quad \forall v_h \in V_h.
    \end{equation}

The following lemma presents an analogue of Céa's lemma for the proposed method.
\begin{LEmma}\label{cealem}
Under the assumptions of Theorem~\ref{coercive}, the solutions $u \in V$ of \eqref{P} and $u_h \in V_h$ of \eqref{FEM} satisfy
    \begin{equation}\label{cea}
        \norme{u-u_h} \lesssim \inf_{v_h \in V_h} \bk{ \normes{u-v_h} + \snormgh{v_h} }.
    \end{equation}
\end{LEmma}
The proof of this lemma follows readily from a standard argument, employing the coercivity of $A_h$ \eqref{Ahcoer}, the weak Galerkin orthogonality \eqref{GaOrth}, and the continuity of $a_h$ \eqref{ahct}. We omit the details here.

To establish the estimate in \eqref{cea}, we need several approximation properties. We first recall the following version of Stein's extension theorem for Lipschitz domains, which will play a key role in our subsequent analysis \cite{Stein}.
\begin{THeorem}\label{extthm}
    There exists a bounded linear operator $E: H^2(Q) \rightarrow H^2(\R^{d+1})$ such that for any $u \in H^2(Q)$, $Eu|_{Q} = u$ and 
    \begin{equation}\label{extineq}
        \norm{Eu}_{2,\R^{d+1}} \leq C_{ext} \norm{u}_{2,Q},
    \end{equation}
    where the constant $C_{ext}$ is dependent of the dimension $d$ and the Lipschitz constant of $Q$.
\end{THeorem}

Let $\pi_h: H^2(\Qact) \rightarrow V_h$ be the Scott-Zhang interpolation operator, which admits the following global approximation estimate \cite{ScottZhang}:
\begin{equation}\label{sz}
    \normx{v-\pi_h v}_{m,\Thact} \lesssim h^{2-m} \snorm{v}_{2,\Qact}, \quad \forall 0 \leq m \leq 2, \quad \forall v \in H^2(\Qact).
\end{equation}
Define the extension-interpolation operator $\pi_h^e : H^2(Q) \rightarrow V_h$ by $\pi_h^e = \pi_h \circ E$. Using the approximation properties of $\pi_h$ \eqref{sz} and the stability of the extension operator $E$ \eqref{extineq}, we obtain 
\begin{equation*}
    \begin{aligned}
    \normx{v-\pi_h^e v}_{m,\Thact \cap Q} 
    &= \normx{E v-\pi_h E v}_{m,\Thact \cap Q} 
    \leq \normx{E v-\pi_h E v}_{m,\Thact} \\ 
    &\lesssim h^{2-m} \normx{E v}_{2,\Qact}
    \lesssim h^{2-m} \normx{v}_{2,Q}.
    \end{aligned}
\end{equation*}
Consequently, the operator $\pi_h^e$ satisfies the global error estimate:
\begin{equation}\label{piheAppro}
    \normx{v-\pi_h^e v}_{m,\Thact \cap Q} \lesssim h^{2-m} \normx{v}_{2,Q}, \quad \forall 0 \leq m \leq 2, \quad \forall v \in H^2(Q).
\end{equation}
We also require the following approximation property of the ghost penalty term $g_h$ \cite{GURKAN2019466}:
\begin{equation}\label{ghAppro}
    \snormgh{\pi_h^e v} \lesssim h \normx{v}_{2,Q}, \quad \forall v \in H^2(Q).
\end{equation}

With these results, we now present the main energy error estimate in the theorem below.
\begin{THeorem} \label{Eerr}
Let $u \in H^2(Q)$ and $u_h \in V_h$ be the solutions of \eqref{P} and \eqref{FEM}, respectively. Under the conditions of Theorem~\ref{coercive}, the following energy error estimate holds:
    \begin{equation} \label{eerr}
        \norme{u-u_h} \lesssim h \normx{u}_{2,Q}.
    \end{equation}
\end{THeorem}
\begin{proof}
Employing Theorem~\ref{extthm}, we extend $u \in H^2(Q)$ to a function in $H^{2}(\R^{d+1})$, again denoted by $u$. Let $\rho = u - \pi_h^e u$. In accordance with the definition of the energy norm, we have
    \begin{equation} \label{III}
    \begin{aligned} 
        \normes{\rho}^2 &= \bk{ \normx{\afrac \na_x \rho}_Q^2 + \normx{\defrac h \rho_t}_Q^2 + \normx{\denegfrac h^{-1} \rho}_Q^2 } + \normx{\defrac h \na_x \cdot (a \na_x \rho)}_{\Thact \cap Q}^2 \\
        &\quad + \bk{ \normxss{\gafrac \hnegfrac \rho}^2 + \normxsz{\rho}^2 + \normxst{\rho}^2 + \normxss{\ganegfrac \hfrac a \na_x \rho \cdot n_x}^2 } \\
        &=: \II_1 + \II_2 + \II_3 .
    \end{aligned}
    \end{equation}
We now estimate the terms $\II_1, \II_2$, and $\II_3$ defined in the previous equation. For $\II_1$, we apply the approximation properties in \eqref{piheAppro} to obtain
    \begin{equation} \label{I1}
    \begin{aligned} 
        \II_1 &\leq a_{max} \normx{\na_x \rho}_Q^2 + \de h^2 \normx{\rho_t}_Q^2 + \de^{-1} h^{-2} \norm{\rho}_Q^2 \\
        &\lesssim h^{2} \norm{u}_{2,Q}^2 + h^{4} \norm{u}_{2,Q}^2 + h^{2} \norm{u}_{2,Q}^2 
        \lesssim h^{2} \norm{u}_{2,Q}^2 .
    \end{aligned}
    \end{equation}
Similarly, the estimate for $\II_2$ follows from \eqref{piheAppro}:
    \begin{equation} \label{I2}
    \begin{aligned}
        \II_2 &\ \lesssim\ \de \normx{h \na_x a \cdot \na_x \rho}_{\Thact \cap Q}^2 + \de \normx{h a \na_x \cdot \na_x \rho}_{\Thact \cap Q}^2 \\
        &\lesssim \de h^2 \normx{\na_x a}_{L^{\infty}(Q)}^2 \normx{\na_x \rho}_Q^2 + \de h^2 a_{max}^2   \norm{\rho}_{2,\Thact \cap Q}^2 \\
        &\lesssim h^{4} \norm{u}_{2,Q}^2 + h^{2} \norm{u}_{2,Q}^2
        \lesssim h^{2} \norm{u}_{2,Q}^2.
    \end{aligned}    
    \end{equation}
To bound $\II_3$, we use the trace inequality \eqref{cutTr}, the approximation properties in \eqref{piheAppro}, and the stability of the extension operator $E$ in \eqref{extineq}:
    \begin{equation} \label{I3}
    \begin{aligned}
        \II_3 &\lesssim  h^{-2}\normx{\rho}_{\Qact}^2 + \normx{\na \rho}_{\Qact}^2 + 
        h^{-1}\normx{\rho}_{\Qact}^2 + h \normx{\na \rho}_{\Qact}^2 \\
        & \quad + a_{max}^2 \bk{ \normx{\na \rho}_{\Qact}^2 + h^2 \normx{\rho}_{2,\Thact}^2 } \\
        & \lesssim h^{2} \norm{u}_{2,\Qact}^2 + h^{3} \norm{u}_{2,\Qact}^2 + h^{2} \norm{u}_{2,\Qact}^2 \\
        & \lesssim h^{2} \norm{u}_{2,\Qact}^2 \lesssim h^{2} \norm{u}_{2,Q}^2. 
    \end{aligned}    
    \end{equation}
Combining the estimates in \eqref{I1}, \eqref{I2}, and \eqref{I3} with the decomposition \eqref{III}, we conclude that
    \begin{equation}\label{rho}
        \normes{u-\pi_h^e u} = \normes{\rho} \lesssim h \norm{u}_{2,Q}.
    \end{equation}
Finally, the desired energy error estimate follows by applying Lemma~\ref{cealem} together with the bounds \eqref{rho} and \eqref{ghAppro}.
\end{proof}

\begin{REmark}
    The $H^2(Q)$ regularity is assumed for convenience in establishing optimal convergence. However, this regularity can be weakened to $H^{1+s}(Q)$ for any $\frac{1}{2} < s \leq 1$. To achieve this estimation, the last term of the definition of $\normes{v}^2$ is replaced by
    \begin{equation*}
        \sum_{K \in \Thact} \bk{ \normx{\defrac \hfrac a\na_x v \cdot n_{\pa K}}_{\pa K \cap Q}^2 +  \normx{\defrac \hfrac a\na_x v \cdot n_x}_{K \cap \pa Q}^2 },
    \end{equation*}
    where $n_{\pa K}$ denotes the unit outward normal vector of $\pa K$. 
    Then Theorem \ref{continuous} can be accordingly derived by applying integration by parts to the last term of $a_h$ and using \eqref{ghH1t}.
    Following the definition of the fractional Sobolev space \cite{fracS} and \cite{szfrac,ghLow}, the results from Lemma \ref{cutlem}, Theorem \ref{extthm} and \eqref{piheAppro} can be extended to fractional order. 
    Proceeding as in Theorem \ref{Eerr}, the energy error estimate can be generalized to
    \begin{equation*}
        \norme{u-u_h} \lesssim h^s \normx{u}_{1+s,Q},
    \end{equation*}
    provided that the exact solution satisfies $u \in H^{1+s}(Q), ~\frac{1}{2} < s \leq 1$.
\end{REmark}

\section{Condition number estimates}\label{Sec5}
This section estimates the condition number of the stiffness matrix and shows its independence of the way the boundary $\Si_s$ intersects the mesh. A key ingredient in our analysis is the following space-time Poincaré–Friedrichs inequality:
 \begin{equation*}
        \norm{u}_Q \leq C \bk{ \norm{\na_x u}_Q + \normxss{u} }, \quad \forall u \in \hioQ,
 \end{equation*}
where the constant $C$ depends only on $d,C_{\Om}$, and $C_{\pa \Om}$. 

\begin{LEmma} \label{lemst1D}
    Let $d=1$. Then for all $u\in \hioQ$, the following inequality holds:
    \begin{equation}\label{st1dineq}
        \norm{u}_Q \leq C \bk{ \norm{\na_x u}_Q + \normxss{u} },
    \end{equation}
    where the constant $C$ depends only on $C_{\Om}$. 
\end{LEmma}
\begin{proof}
    Assume first that $\Om_t=(a(t),b(t))$ is an interval with $a(t) < b(t)$ for each fixed $t$. By \cite[Theorem 3.1]{PF1d} we have
    \begin{equation} \label{interval}
        \norm{w}_{\Om_t}^2 \leq 2|\Om_t|^2 \norm{\pa_x w}_{\Om_t}^2 + 2|\Om_t|(T_0w(a(t)))^2, \quad \forall w \in H^1(\Om_t),
    \end{equation}
    where $T_0: H^1(\Om_t) \rightarrow L^2(\pa \Om_t)$ denotes the trace operator.
    
    Now suppose $\Omega_t$ satisfies Assumption \ref{assump1} but is not necessarily an interval. Since $\Omega_t$ is a bounded open set in $\R$, it follows from \cite[Theorem 1]{openset} that $\Omega_t$ is a union of countably many disjoint open intervals. Moreover, by Assumption \ref{assump1}, we have $\snorm{\partial \Omega_t} \leq C_{\partial \Omega}$, which implies that $\Omega_t$ consists of only finitely many such intervals. Using the bound $\snorm{\Omega_t} \leq C_{\Omega}$ together with \eqref{interval}, we obtain
    \begin{equation} \label{aux1d}
    	\norm{w}_{\Omega_t}^2 \leq 2C_{\Omega}^2 \norm{\pa_x w}_{\Omega_t}^2 + 2C_{\Omega} \norm{w}_{\partial \Omega_t}^2, \quad \forall w \in H^1(\Omega_t).
    \end{equation}
    
    Now take any $u \in \ciQ$. Applying \eqref{aux1d} gives
    \begin{equation}\label{1d}
    	\norm{u}_{\Omega_t}^2 \leq 2C_{\Omega}^2 \norm{\na_x u}_{\Omega_t}^2 + 2C_{\Omega} \norm{u}_{\partial \Omega_t}^2, \quad \forall t \in [0,T].
    \end{equation}
    Integrating \eqref{1d} over $t \in [0,T]$ yields
    \begin{equation} \label{C1ineq}
    	\norm{u}_{Q}^2 \leq 2C_{\Omega}^2 \norm{\na_x u}_{Q}^2 + 2C_{\Omega} \normxss{u}^2.
    \end{equation}
    
    Under Assumption \ref{assump2}, the space $\ciQ$ is dense in $\hioQ$ and the trace operator $\gamma_s$ is continuous. Therefore, inequality \eqref{C1ineq} extends to all $u \in \hioQ$. Taking $C = \sqrt{2} \max\left\{C_{\Omega}, \sqrt{C_{\Omega}}\right\}$ completes the proof of \eqref{st1dineq}. \end{proof}

For $d \geq 2$, the following Sobolev inequality will be needed.
\begin{THeorem} \label{PFineq}
Let $\Omega \subset \R^d$ with $d \geq 2$ be a bounded Lipschitz domain satisfying $\snorm{\partial \Omega} < \infty$. Then there exists a constant $C(d)$, depending only on $d$, such that for all $u \in H^1(\Omega)$,
    \begin{equation}\label{st2dineq}
        \norm{u}_{\Om} \leq  C(d) \bk{ \snorm{\Om}^{\frac{1}{d}} \norm{\na u}_{\Om} + \snorm{\pa \Om}^{\frac{1}{2d-2}} \norm{u}_{\pa \Om} }.
    \end{equation}
\end{THeorem}
\begin{proof}
    By \cite[Theorem 1.2]{maggi2005}, for $d\geq 2$, and exponent $p \in \left[1,d\right)$, if a function $f:\Om \rightarrow \R$ satisfies $\norm{\na f}_{L^p(\Om)}, \norm{f}_{L^{p^{\sharp}}(\pa \Om)} < \infty$, then
    \begin{equation} \label{pspsharp}
        \norm{f}_{L^{p^*}(\Om)} \leq C_1(d,p) \norm{\na f}_{L^p(\Om)} + C_2(d,p) \norm{f}_{L^{p^{\sharp}}(\pa \Om)},
    \end{equation}
    where $p^* = \frac{d p}{d-p}$, $p^{\sharp} = \frac{(d-1)p}{d-p}$, and the constants $C_1(d,p), C_2(d,p)$ depend only on $d$ and $p$. 
  
   Now choose $p = \frac{2d}{d+2}$, which lies in $[1, 2)$. Then we have $p^* = 2$ and $p^{\sharp} = \frac{2(d-1)}{d} < 2$. Define $r = \frac{2}{p}$ and $s = \frac{2}{2 - p}$, so that $\frac{1}{r} + \frac{1}{s} = 1$. By H\"older's inequality, 
    \begin{equation*} 
        \norm{\na u}_{L^p(\Om)}^p = \int_{\Om} \snorm{\na u}^p \cdot 1 \rd x 
        \leq \bk{ \int_{\Om} \bk{\snorm{\na u}^p}^r \rd x}^{\frac{1}{r}} \bk{ \int_{\Om} 1^s \rd x}^{\frac{1}{s}} 
        = \snorm{\Om}^{1-\frac{p}{2}} \norm{\na u}_{L^2(\Om)}^p,
    \end{equation*}
which implies
    \begin{equation*}
        \norm{\na u}_{L^p(\Om)} \leq \snorm{\Om}^{\frac{1}{p}-\frac{1}{2}} \norm{\na u}_{\Om} = \snorm{\Om}^{\frac{1}{d}} \norm{\na u}_{\Om}.
    \end{equation*}
A similar argument gives
    \begin{equation*}
        \norm{u}_{L^{p^{\sharp}}(\pa\Om)} \leq \snorm{\pa \Om}^{\frac{1}{p^{\sharp}}-\frac{1}{2}} \norm{u}_{\pa \Om} = \snorm{\pa \Om}^{\frac{1}{2d-2}} \norm{u}_{\pa \Om}.
    \end{equation*}
Substituting these estimates into \eqref{pspsharp} completes the proof of \eqref{st2dineq}.
\end{proof}

\begin{REmark}
    Theorem \ref{PFineq} imposes only mild assumptions on the domain $\Om$, requiring merely that it be a bounded Lipschitz domain with boundary of finite Lebesgue measure. Moreover, the generic constant in \eqref{st2dineq} depends only on the dimension $d$ and is independent of whether $\Om$ is convex or possesses a rough boundary. This follows from the application of the auxiliary result in \cite[Theorem 1.2]{maggi2005}, where the Sobolev inequality \eqref{pspsharp} is shown to hold even for locally Lipschitz domains that are not necessarily bounded.
\end{REmark}

We now extend the space-time Poincaré–Friedrichs inequality to dimensions $d \geq 2$ using an argument analogous to that in Lemma~\ref{lemst1D}.
\begin{LEmma}\label{lemst2D}
Let $d \geq 2$. Then for all $u \in \hioQ$,
\begin{equation}\label{stineq}
	\norm{u}_Q \leq C \left( \norm{\na_x u}_Q + \normxss{u} \right),
\end{equation}
where the constant $C$ depends only on $d, C_{\Omega}$, and $C_{\partial \Omega}$.
\end{LEmma}

For the condition number analysis, we require the $L^2$ stability of the ghost penalty $g_h$, which is stated as follows \cite{GURKAN2019466}:
\begin{equation}\label{ghstab}
    \snormgh{v_h} \lesssim h^{-1} \normx{v_h}_{\Qact}, \quad \forall v_h \in V_h.
\end{equation}

\begin{LEmma}
    For any $v_h \in V_h$, it holds that
    \begin{align} 
        \normx{v_h}_{\Qact} &\lesssim \normeh{v_h}, \label{l2toE}\\ 
        \normehs{v_h} &\lesssim h^{-1} \normx{v_h}_{\Qact}.\label{Etol2}
    \end{align}
\end{LEmma}
\begin{proof}
    From \eqref{ghL2} and \eqref{stineq}, it follows that
    \begin{align*}
        \normx{v_h}_{\Qact} &\lesssim \norm{v_h}_Q + \snormgh{v_h} 
        \lesssim \norm{\na_x v_h}_Q + \normxss{v_h} + \snormgh{v_h} \\
        &\lesssim \normx{\afrac \na_x v_h}_Q + \normxss{\gafrac \hnegfrac v_h} + \snormgh{v_h}
        \leq \normeh{v_h},
    \end{align*}
    which yields \eqref{l2toE}. To prove \eqref{Etol2}, we begin by observing its definition:
    \begin{equation} \label{ii}
    \begin{aligned}
        \normehs{v_h}^2 &= \bk{ \normx{\afrac \na_x v_h}_Q^2 + \normx{\defrac h \pa_t v_h}_Q^2 + \normx{\denegfrac h^{-1}v_h}_Q^2 + \normx{\defrac h\na_x a \cdot \na_x v_h}_{\Thact \cap Q}^2 } \\
        & \quad + \bk{ \normxss{\ganegfrac \hfrac a\na_x v_h \cdot n_x}^2 + \normxss{\gafrac \hnegfrac v_h}^2 + \normxsz{v_h}^2 + \normxst{v_h}^2 + \snormgh{v_h}^2 } \\
        & =: \RR_1 + \RR_2.
    \end{aligned}
    \end{equation}
To bound $\RR_1$, we apply a standard finite element inverse estimate, which gives
    \begin{equation} \label{i1}
    \begin{aligned}
        \RR_1 &\leq \normx{a}_{L^{\infty}(Q)} \normx{\na_x v_h}_Q^2 + \de h^2 \normx{\pa_t v_h}_Q^2 + \de^{-1} h^{-2} \normx{v_h}_Q^2 +
        \de h^2 \normx{\na_x a}_{L^{\infty}(Q)}^2 \normx{\na_x v_h}_Q^2 \\
        &\lesssim h^{-2} \normx{v_h}_{\Qact}^2 + \normx{v_h}_{\Qact}^2 + h^{-2} \normx{v_h}_{Q}^2 + h^{-2} \normx{v_h}_{\Qact}^2 
        \lesssim  h^{-2} \normx{v_h}_{\Qact}^2.
    \end{aligned}
    \end{equation}
    To estimate $\RR_2$, we apply the inverse equality \eqref{cutTrpoly}, the finite element inverse estimate, and the $L^2$ stability of ghost penalty \eqref{ghstab}, which together yield
    \begin{equation} \label{i2}
    \begin{aligned}
        \RR_2 &\leq \ga^{-1} h \normx{a}_{L^{\infty}(\Si_s)}^2 \normxss{\na v_h}^2 + \ga h^{-1} \normxss{v_h}^2 + \normxsz{v_h}^2 + \normxst{v_h}^2 + \snormgh{v_h}^2 \\
        &\lesssim \normx{\na v_h}_{\Qact}^2 + h^{-2} \normx{v_h}_{\Qact}^2 + h^{-1}\normx{v_h}_{\Qact}^2 + h^{-2}\normx{v_h}_{\Qact}^2
        \lesssim h^{-2}\normx{v_h}_{\Qact}^2.
    \end{aligned}  
    \end{equation}
    The proof is completed by combining \eqref{ii}, \eqref{i1} and \eqref{i2}.
\end{proof}

Thus, the condition number of the stiffness matrix can be bounded as stated in the following theorem.

\begin{THeorem} \label{cdnum}
    Let $A$ be the stiffness matrix of the finite element method \eqref{FEM}. Under the assumptions of Theorem~\ref{coercive}, we have
    \begin{equation*}
        \kappa(A) \lesssim h^{-2},
    \end{equation*}
    where $\kappa(A) = \normx{A}\normx{A^{-1}}$ with $\normx{\cdot}$ being the Euclidean matrix norm.
\end{THeorem}
\begin{proof}
The result follows directly from \cite[Theorem 2.16]{GURKAN2019466} combined with the norm equivalence relations \eqref{l2toE} and \eqref{Etol2}.
\end{proof}

\section{Numerical experiments}\label{Sec6}
In this section, we perform several numerical experiments to validate the theoretical analysis of the proposed method. In all examples, the moving domain is defined via a level set function $\phi(x,t)$ as $\Om(t) = \{ x \in \mathbb{R}^{d}: \phi(x,t) < 0 \}$. The relative errors in the $H^{1,0}$ and $L^2$ norms are respectively defined as
\begin{equation*}
    \frac{\normx{\afrac \na_x (u-u_h)}_Q}{\normx{\afrac \na_x u}_Q}, \quad 
    \frac{\normx{u-u_h}_Q}{\normx{u}_Q}.
\end{equation*}
Based on the error estimate \eqref{eerr}, the $H^{1,0}$-norm error is expected to satisfy $\normx{\afrac \na_x (u-u_h)}_{Q} \lesssim h \normx{u}_{2,Q}$ for any sufficiently smooth solution $u$. 

To handle numerical integration over curved domains, we adopt the approach from \cite{FRIES2017759}. First, a set of sample points is distributed within each element to identify the cut elements. For each cut element, a recursive subdivision algorithm is applied to generate valid subelements for numerical integration. On the uncut subelements, a standard Gaussian quadrature rule is employed for volume integration. For the valid cut subelements, two isoparametric mappings are constructed to define the surface and volume quadrature rules, respectively. We refer to \cite{FRIES2017759} for more details. In our experiments, we employ the quadratic isoparametric mappings and the fourth-order Gaussian quadrature rule for both surface and volume integration to preserve accuracy.

\subsection{Convergence studies for 1D case}
We present two 1D examples to examine the convergence behavior of the proposed method. The first is inspired by a Stefan-type problem, while the second involves a moving interval with a time-dependent diffusion coefficient. 
\begin{EXample} \label{ex1}
This example is adapted from a Stefan-type problem \cite{stefan}. The domain $\Om(t) = [0,s(t)]$ features a fixed left boundary and a moving right boundary defined by $s(t) = 2\al \sqrt{t+t_0}$, where $\al, t_0>0$ are constants (see Figure \ref{stefan} for an illustration). The corresponding level set function is given by
\begin{equation*}
    \phi(x,t) = x(x-s(t)).
\end{equation*}
We set $a(x,t) \equiv 1, f = 0, u(0,t)=1, u(s(t),t)=0, \al = 0.5, t_0 = 1.2$, and the initial data $u_0$ is chosen such that the exact solution  takes the form
\begin{equation*}
    u(x,t) = 1-\frac{1}{\text{erf}(\al)} \text{erf} \bk{ \frac{x}{2\sqrt{t+t_0}} },
\end{equation*}
where $\text{erf}(x)$ denotes the error function,
\begin{equation*}
    \text{erf}(x) = \frac{2}{\sqrt{\pi}} \int_0^x e^{-s^2} \rd s.
\end{equation*}
The final time is set to $T = 1$. 
\end{EXample}

We take the background domain as $\wOm = (0,1.5)$ and construct the initial space-time mesh $\mathcal{T}_h$ of $\wQ$ as follows. The spatial domain $\wQ$ and the time interval $I = (0, T)$ are first divided into $N_x = 12$ and $N_t = 8$ uniform subintervals, respectively, forming an $N_x \times N_t$ Cartesian grid. Each rectangle is then subdivided into two triangles to form a conforming triangulation (see Figure~\ref{stefan}). This initial mesh $\Th$ is subsequently refined uniformly four times, generating a sequence of meshes $\set{ \Th, \mathcal{T}_{h/2}, \mathcal{T}_{h/4}, \mathcal{T}_{h/8}, \mathcal{T}_{h/16} }$  by dividing each triangle into four congruent subtriangles at each refinement.

In the experiment, we choose the parameters  $\ga = 50, \ga_1 = 0.1$ and $\de = 0.2$. The relative $H^{1,0}$-and $L^2$-errors are computed and summarized in Table~\ref{ex1tab}. The results indicate that the $H^{1,0}$-error converges optimally at a rate of $O(h)$, which aligns with the theoretical analysis. We also observe an optimal convergence rate of $O(h^2)$ for the $L^2$-error. Furthermore, the condition number of the stiffness matrix is computed across different mesh sizes; as illustrated in Figure~\ref{stefan_err}, it scales as $O(h^{-2})$, confirming the theoretical prediction of Theorem~\ref{cdnum}.

\begin{figure}[htbp] 
\centering
	\includegraphics[width = 0.32\textwidth]{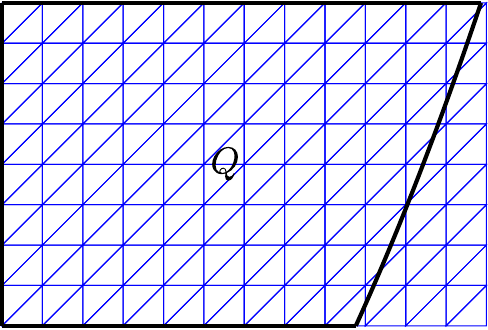} 
	\caption{The moving domain and the initial mesh of Example \ref{ex1}.}
    \label{stefan}
\end{figure} 

 \begin{figure}[htbp]
  \centering
  \begin{minipage}[c][5cm][t]{0.45\textwidth}
    \centering
    	\begin{tabular}{lllll}
        \toprule
        $h$ & $H^{1,0}$ error & order & $L^2$ error & order \\
        \midrule
        1/8   & 1.90E-02 & -     & 6.40E-04 & -     \\
		1/16  & 9.51E-03 & 1.00  & 1.74E-04 & 1.88  \\
		1/32  & 4.75E-03 & 1.00  & 4.72E-05 & 1.88  \\
		1/64  & 2.38E-03 & 1.00  & 1.22E-05 & 1.95  \\
		1/128 & 1.19E-03 & 1.00  & 3.11E-06 & 1.97  \\
       \bottomrule
	\end{tabular}
	\captionof{table}{Numerical errors of Example \ref{ex1}}
	\label{ex1tab}
  \end{minipage}
    \hfill 
  \begin{minipage}[c][5cm][t]{0.45\textwidth} 
    \centering
    \includegraphics[width=0.8\textwidth]{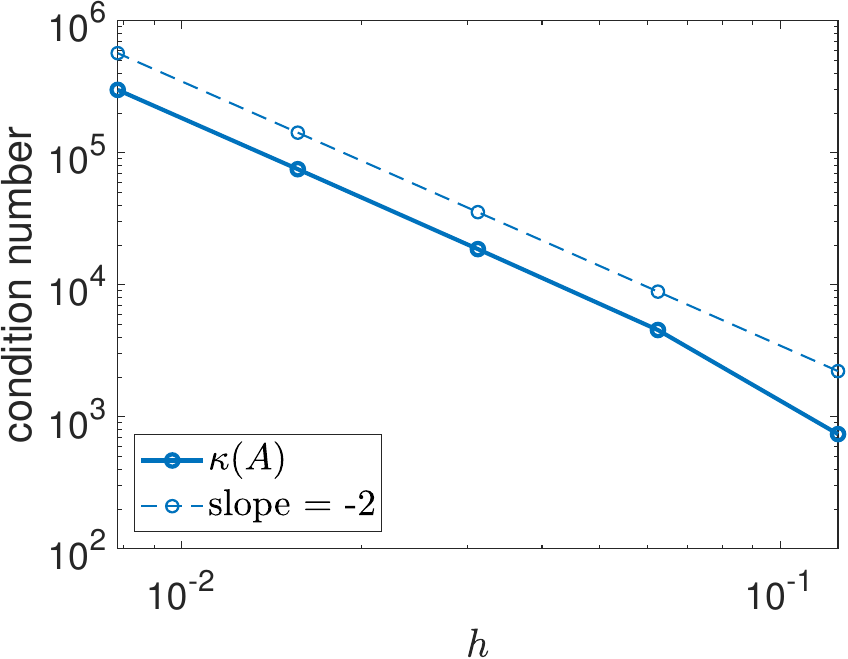} 
    \caption{The condition number of the stiffness matrix of Example \ref{ex1}.}
    \label{stefan_err}
  \end{minipage}
   \end{figure}

\begin{EXample} \label{ex2}
    Consider a moving interval with velocity $v(t)$, described by the level set function
\begin{equation*}
    \phi(x,t) = (x-a_0-v(t))(x-b_0-v(t)),
\end{equation*}
where $a_0 = 0.3, b_0 = 0.7$ and $v(t) = \pi \sin(2\pi t)/20$. The final time is set to $T = 1$. This velocity function causes the interval to oscillate periodically over time, as depicted in Figure~\ref{Sdomain}. The time-dependent diffusion coefficient is given by $a(x,t) = 0.5t \cos^2(x) + 0.1$. The source term $f$, boundary data $g$ and initial data $u_0$ are prescribed such that the exact solution takes the form
\begin{equation*}
    u(x,t) = \bk{\sin(2\pi x)\sin(2\pi t) + \cos(2\pi x)\cos(2\pi t)}e^{-t}.
\end{equation*}
\end{EXample}
The background domain is taken as  $\wOm = (0,1)$. The initial background mesh and its subsequent refinements are constructed following the same procedure as in Example~\ref{ex1}, with $N_x = N_t = 14$; see Figure~\ref{Sdomain} for an illustration. The parameters are chosen as $\ga = 100, \ga_1 = 0.1$ and $\de = 0.1$. 

The relative $H^{1,0}$-error, $L^2$-error, and the condition number of the stiffness matrix are summarized in Table~\ref{ex2tab} and visualized in Figure~\ref{Sdomain_err} respectively. Results show that the $H^{1,0}$-error again achieves the optimal convergence rate of $O(h)$, while the $L^2$-error converges at a rate even exceeding $O(h^2)$. In addition, the condition number grows as $O(h^{-2})$, which is consistent with our theoretical analysis.

\begin{figure}[htbp]
    \centering
	\includegraphics[height = 0.3\textwidth]{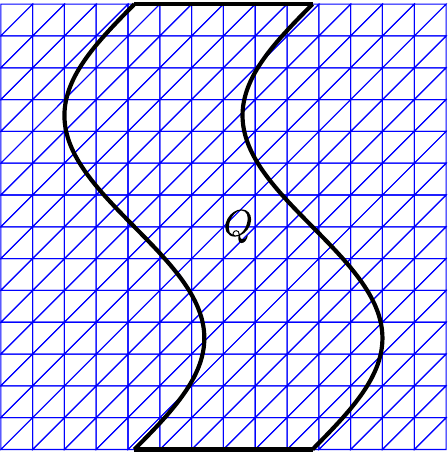}
	\caption{The moving domain and the initial mesh of Example \ref{ex2}.}
    \label{Sdomain}
\end{figure}

\begin{figure}[htbp]
  \centering
  \begin{minipage}[c][5.7cm][t]{0.45\textwidth}
    \centering
    	\begin{tabular}{lllll}
        \toprule
        $h$ & $H^{1,0}$ error & order & $L^2$ error & order \\
        \midrule
        1/14   & 1.75E-01 & -     & 5.68E-02 & -     \\
		1/28   & 7.69E-02 & 1.18  & 1.17E-02 & 2.28  \\
		1/56   & 3.66E-02 & 1.07  & 2.04E-03 & 2.52  \\
		1/112  & 1.80E-02 & 1.02  & 3.86E-04 & 2.41  \\
		1/224  & 8.94E-03 & 1.01  & 7.44E-05 & 2.37  \\
       \bottomrule
	\end{tabular}
	\captionof{table}{Numerical errors of Example \ref{ex2}}
	\label{ex2tab}
  \end{minipage}
    \hfill 
  \begin{minipage}[c][5.7cm][t]{0.45\textwidth} 
    \centering
    \includegraphics[width=0.8\textwidth]{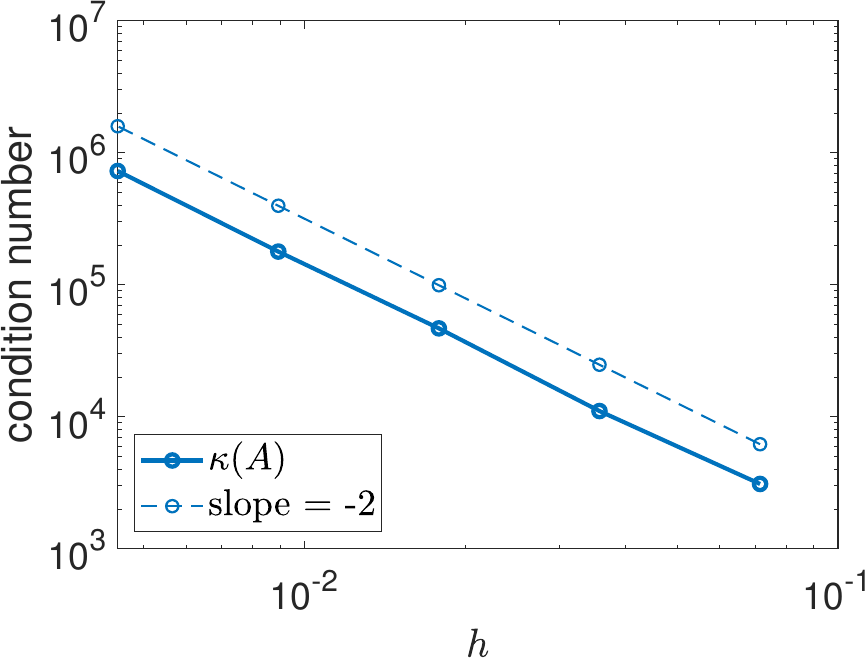} 
    \caption{The condition number of the stiffness matrix of Example \ref{ex2}.}
    \label{Sdomain_err}
  \end{minipage}
   \end{figure}

\subsection{Convergence studies for 2D case}
We now present 2D examples to evaluate the convergence behavior of the proposed method. The first involves a moving circle with a constant diffusion coefficient, while the second considers a complex geometry with a time-dependent variable diffusion coefficient. 
\begin{EXample} \label{ex3}
Consider a moving circular domain defined by the level set function
\begin{equation*}
    \phi(x,y,t) = (x-0.15\cos(2\pi t)-x_0)^2 + (y-0.15\sin(2\pi t)-y_0)^2 - r_0^2,
\end{equation*}
with $x_0 = 0.5, y_0 = 0.5, r_0 = \pi/12$, and the final time is set as $T = 1$. We set the diffusion coefficient $a(x,t) \equiv 1$, and define $f$, $g$, and $u_0$ such that the exact solution is given by
\begin{equation} \label{ue3d}
    u(x,y,t) = \bk{\sin(2\pi x)\sin(2\pi y)\sin(2\pi t) + \cos(2\pi x)\cos(2\pi y)\cos(2\pi t)}e^{-t}.
\end{equation}
\end{EXample}
The background domain is taken as $\wOm = (0,1)^2$. The initial space-time mesh is constructed by first dividing $\wQ$ into $N_x \times N_y \times N_t$ cuboids with $N_x = N_y = N_t = 3$, and then subdividing each cuboid into six tetrahedra (see Figure~\ref{circle}). This initial mesh is uniformly refined twice to obtain a finer triangulation $\Th$, where each tetrahedron is split into eight subtetrahedra. We take $\Th$ as the initial computational mesh and apply three further uniform refinements to generate the sequence $\set{ \Th, \mathcal{T}_{h/2}, \mathcal{T}_{h/4}, \mathcal{T}_{h/8} }$.

The parameters are chosen as $\gamma = 50, \gamma_1 = 0.1$, and $\delta = 0.2$. On each mesh in the sequence, we compute the relative $H^{1,0}$-error, $L^2$-error, and the condition number of the stiffness matrix. The results are summarized in Table~\ref{ex3tab} and visualized in Figure~\ref{circle_err}. Numerical results confirm that the $H^{1,0}$-error converges optimally at a rate of $O(h)$, while the $L^2$-error exhibits an $O(h^2)$ convergence rate. Furthermore, for sufficiently small mesh sizes, the condition number grows as $O(h^{-2})$, in agreement with the theoretical estimate.

\begin{figure}[htbp]
    \centering
	\includegraphics[width = 0.34\textwidth]{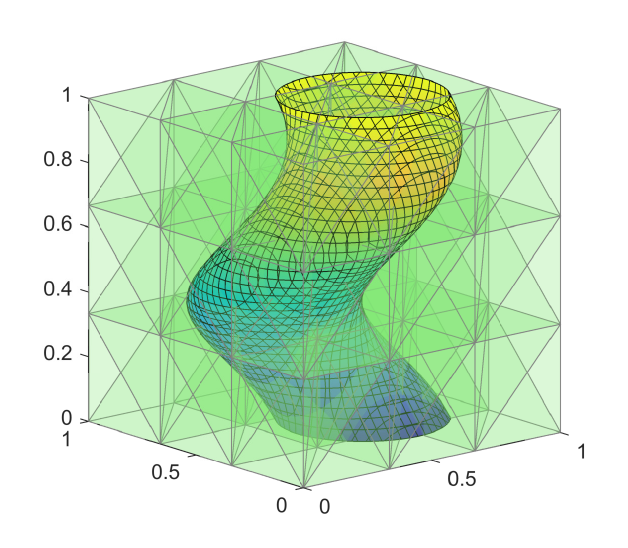}
	\caption{The moving domain and the initial mesh of Example \ref{ex3}.}
     \label{circle}
\end{figure}

\begin{figure}[htbp]
  \centering
  \begin{minipage}[c][5cm][t]{0.45\textwidth}
    \centering
    	\begin{tabular}{lllll}
        \toprule
        $h$ & $H^{1,0}$ error & order & $L^2$ error & order \\
        \midrule
        1/12 & 3.20E-01 & -     & 1.14E-01 & -     \\
		1/24 & 1.58E-01 & 1.02  & 2.91E-02 & 1.97  \\
		1/48 & 7.88E-02 & 1.01  & 7.25E-03 & 2.00  \\
		1/96 & 3.94E-02 & 1.00  & 1.82E-03 & 1.99  \\
       \bottomrule
	\end{tabular}
	\captionof{table}{Numerical errors of Example \ref{ex3}}
	\label{ex3tab}
  \end{minipage}
    \hfill 
  \begin{minipage}[c][5.3cm][t]{0.45\textwidth} 
    \centering
    \includegraphics[width=0.8\textwidth]{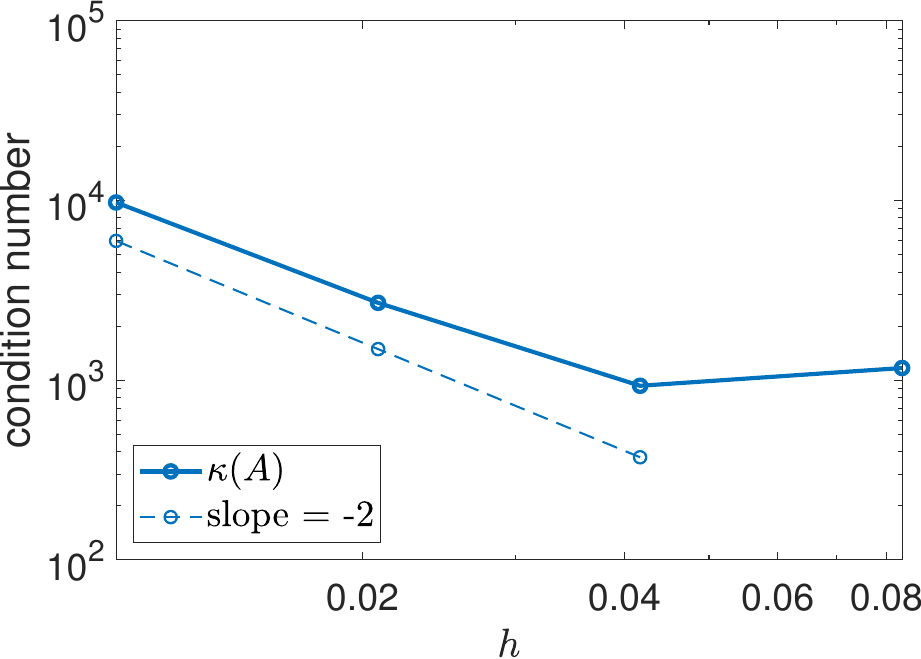} 
    \caption{The condition number of the stiffness matrix of Example \ref{ex3}.}
    \label{circle_err}
  \end{minipage}
   \end{figure}

\begin{EXample} \label{ex4}
    Consider a flower-shaped domain with a central hole, defined by the following level set functions. Let
\begin{equation*}
    X(x,t) = x - 0.1 \cos(\pi t), \quad Y(y,t) = y - 0.1 \sin(\pi t), \quad \theta = \arctan \left( \frac{Y(y,t)}{X(x,t)} \right),
\end{equation*}
and define
\begin{align*}
    \phi_1(x,y,t)& = X(x,t)^2 + Y(y,t)^2 - r^2 + \frac{r}{r_0} \cos(5\theta) \cos \bk{ \frac{2}{3}\pi t},\\
    \phi_2(x,y,t) &= X(x,t)^2 + Y(y,t)^2 - \left( r - \frac{1}{8}t^2 - \frac{\pi}{10} \right)^2,
\end{align*}
with the combined level set function
\begin{equation*}
    \phi(x,y,t) = \phi_1(x,y,t) \phi_2(x,y,t),
\end{equation*}
where $r = 0.7$ and $r_0 = 3.5$. The final time is set to $T = 1$. 

As illustrated in Figure~\ref{flower}, the flower-shaped domain undergoes both rotation and deformation over time, while the circular hole at its center gradually shrinks. The diffusion coefficient is taken as $a(x,t) = 0.5\sin^2(x y t) + 0.1$. The source term $f$, boundary data $g$, and initial value $u_0$ are chosen such that the exact solution coincides with that given in \eqref{ue3d}. 
\end{EXample}
The background domain is taken as $\wOm = (-1,1)^2$. The initial mesh and its refinement sequence are generated following the same procedure as in Example~\ref{ex3}, with $N_x=N_y=4, N_t=2$, as illustrated in Figure~\ref{flower}. We set the parameters as  $\ga = 100, \ga_1 = 0.1$ and $\de = 0.02$, and compute the relative $H^{1,0}$ and $L^2$ error, and the condition number $\kappa(A)$ of the stiffness matrix. The results are summarized in Table~\ref{ex4tab} and Figure~\ref{flower_err}.
Even with the complex motion of the physical domain, the proposed method still achieves optimal convergence rates—$O(h)$ for the $H^{1,0}$-error and $O(h^2)$ for the $L^2$-error—along with an optimal $O(h^{-2})$ growth in the condition number for sufficiently small mesh sizes.

\begin{figure}[htbp]
    \centering
	\includegraphics[width = 0.5\textwidth]{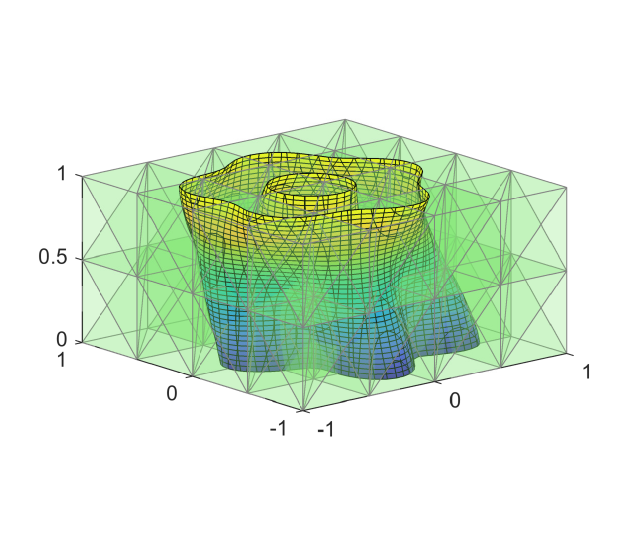}
	\caption{The moving domain and the initial mesh of Example \ref{ex4}.}
     \label{flower}
\end{figure}

\begin{figure}[htbp]
  \centering
  \begin{minipage}[c][5cm][t]{0.45\textwidth}
    \centering
    	\begin{tabular}{lllll}
        \toprule
        $h$ & $H^{1,0}$ error & order & $L^2$ error & order \\
        \midrule
        1/8  & 6.38E-01 & -     & 3.96E-01 & -     \\
		1/16 & 3.18E-01 & 1.00  & 1.41E-01 & 1.49  \\
		1/32 & 1.41E-01 & 1.18  & 3.32E-02 & 2.08  \\
		1/64 & 6.48E-02 & 1.12  & 6.28E-03 & 2.40  \\
       \bottomrule
	\end{tabular}
	\captionof{table}{Numerical errors of Example \ref{ex4}}
	\label{ex4tab}
  \end{minipage}
    \hfill 
  \begin{minipage}[c][5.3cm][t]{0.45\textwidth} 
    \centering
    \includegraphics[width=0.8\textwidth]{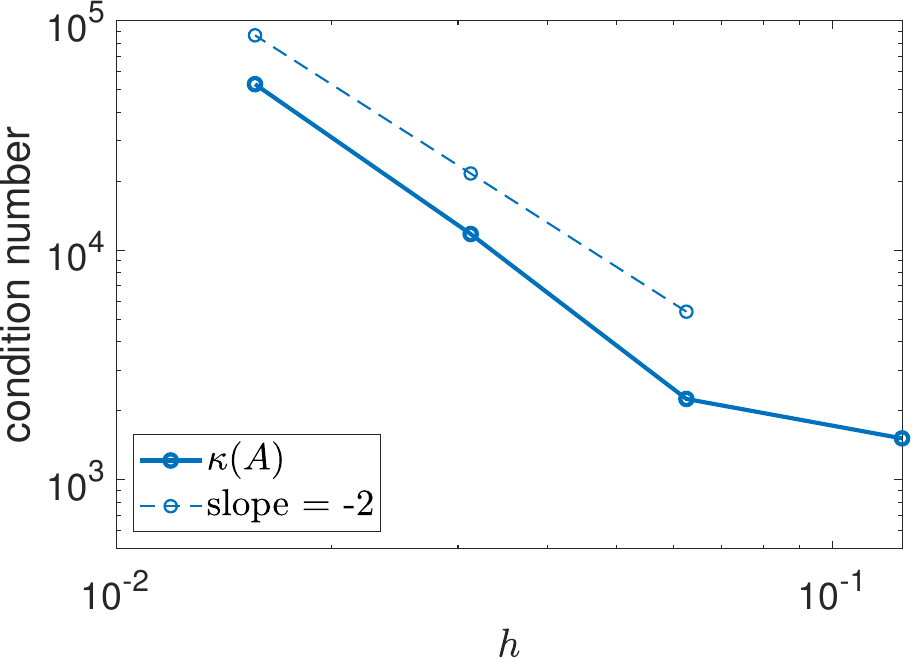} 
    \caption{The condition number of the stiffness matrix of Example \ref{ex4}.}
     \label{flower_err}
  \end{minipage}
   \end{figure}

\subsection{Small cut effect studies}
This subsection examines the sensitivity of the proposed method to small cut configurations. We revisit the moving interval defined in Example~\ref{ex2}, with parameters $a_0 = \pi/24, b_0 = \pi/6$, and a translational motion given by $v(t) = l + t/7$, where $l$ controls the initial shift of the domain. The final time is set to $T = 1$, and the background domain is taken as $\wOm = (0,1)$. The background mesh $\Th$ is constructed following the procedure in Example~\ref{ex2} using $N_x = N_t = 9$.

Varying the shift $l$ alters the relative position between the space-time domain $Q$ and the background mesh $\Th$, potentially leading to small-cut scenarios. We consider $l = \{10^{-5}j\}_{j=0}^{30000}$ and solve the parabolic moving-domain problem \eqref{P} for each value of $l$, using the same source term $f$, boundary data $g$, initial condition $u_0$, and exact solution as in Example~\ref{ex2}. The parameters are chosen as $\gamma = 50$ and $\delta = 0.2$. Two stabilization settings are compared: one with ghost penalty stabilization ($\gamma_1 = 0.1$) and one without any stabilization ($\gamma_1 = 0$). For both cases, we compute the condition number $\kappa(A)$ and the relative errors in the $H^{1,0}$- and $L^2$-norms. The results are presented in Figure~\ref{smallcut}.

\begin{figure}[htbp]
	\centering
	\includegraphics[width = 0.23\textwidth]{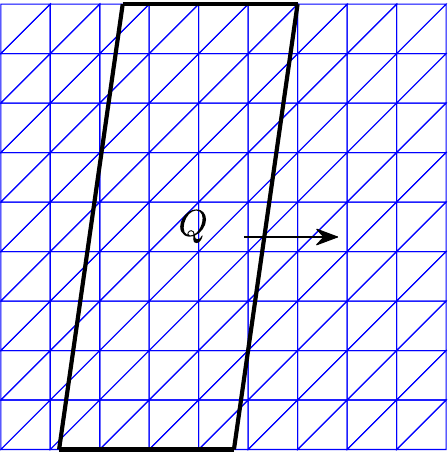}
	\caption{The background mesh $\Th$ and space-time domain $Q$. As we increase $l$, $Q$ moves towards the right direction.}
\end{figure} 

Figure \ref{smallcut}(left) clearly shows that without stabilization, the condition number is highly sensitive to the relative position between \(Q\) and $\Th$, varying from \(10^3\) to about \(10^{14}\). In contrast, with ghost penalty stabilization, \(\kappa(A)\) remains stable and consistently bounded across different values of \(l\). This confirms that the stabilized method is robust with respect to the cut position of the domain relative to the background mesh.

Figure \ref{smallcut}(right) further illustrates that in the absence of stabilization, both the $H^{1,0}$- and \(L^2\)-errors exhibit large fluctuations as \(l\) varies—a likely consequence of ill-conditioning. When ghost penalty is applied, both errors behave stably with only minimal variation.

\begin{figure}[H]
	\makebox[\textwidth][c]
	{\subfigure{\includegraphics[width = 0.34\textwidth]{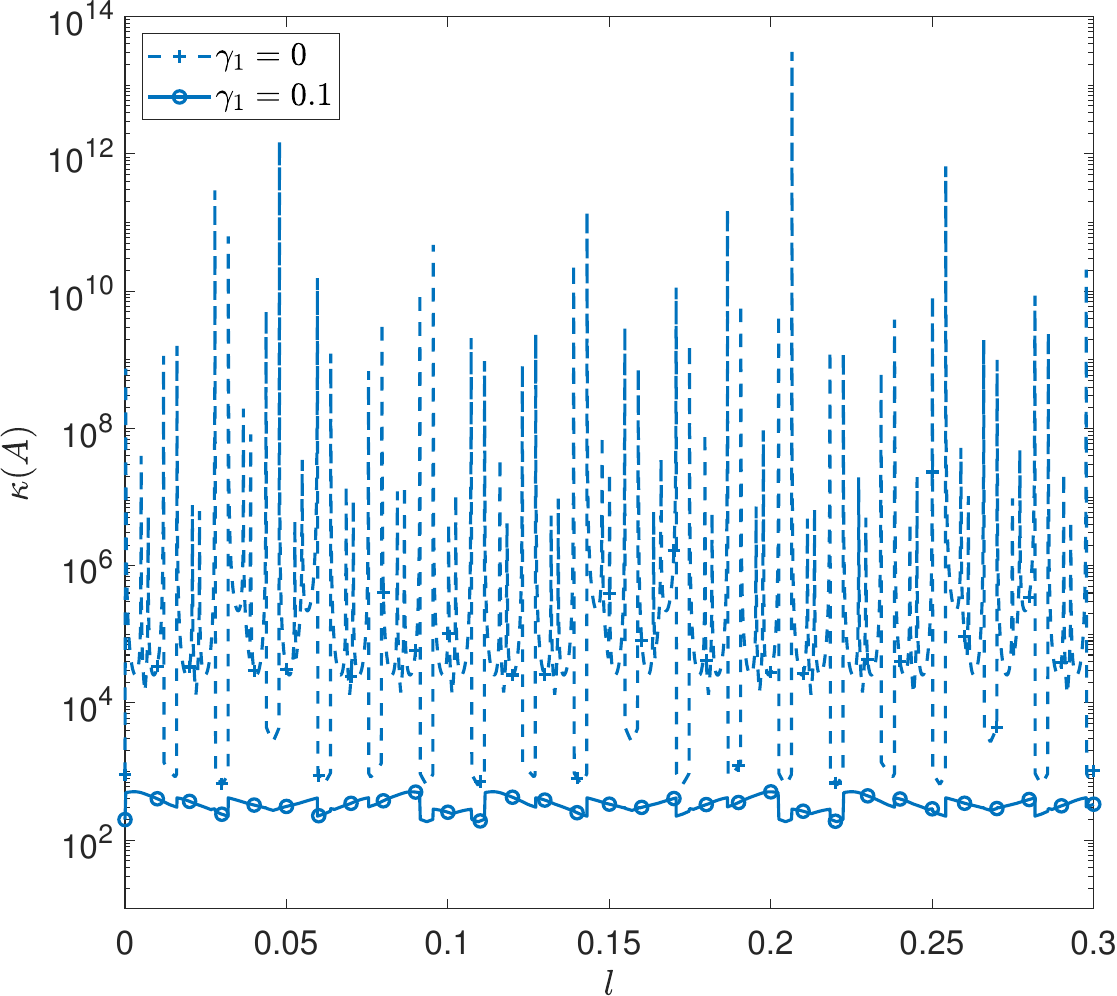}} \hspace{10mm}
	\subfigure{\includegraphics[width = 0.34\textwidth]{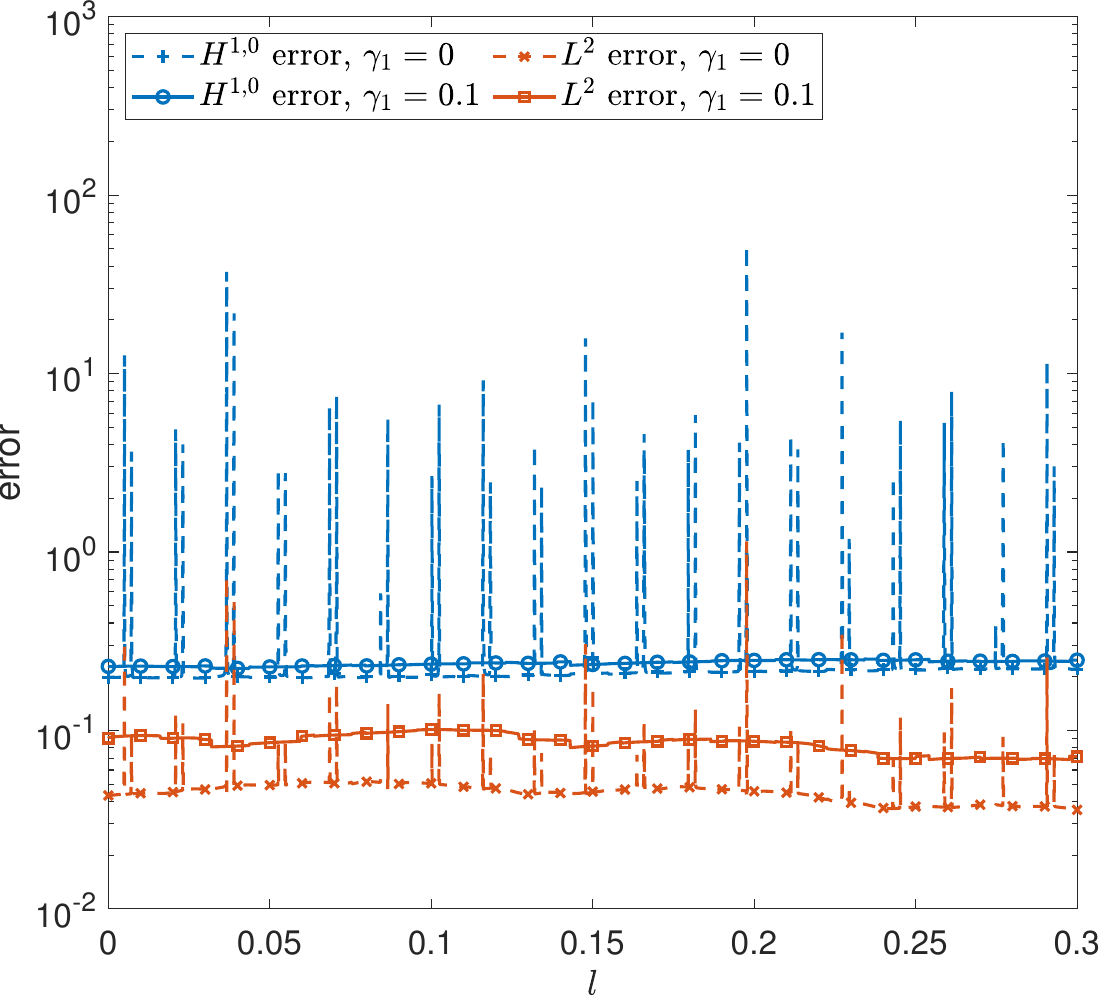}}} 
	\caption{Left: The condition number of the stiffness matrix. Right: The relative error.}
     \label{smallcut}
\end{figure}

\subsection{Boundary layer problem}
It is well known that convection-dominated convection-diffusion problems often develop boundary layers, where the solution exhibits sharp gradients. In such cases, standard FEMs typically deliver suboptimal accuracy and tend to produce spurious oscillations near the boundaries \cite{SOLD}. The streamline upwind Petrov–Galerkin (SUPG) method, on the other hand, has been widely recognized as an effective stabilization technique that suppresses these oscillations by introducing numerical diffusion along the streamlines \cite{supg}.

Interpreting the parabolic equation \eqref{P} as a convection-diffusion problem in the space-time setting, it is natural to expect similar challenges when $a(x,t) \ll 1$. In this subsection, we demonstrate the capability of our method in handling a boundary layer problem, showing that the incorporated SUPG stabilization successfully eliminates numerical oscillations while maintaining accuracy.

We consider a 1D boundary layer problem with the same domain configuration as in Example~\ref{ex2}. Let $a(x,t) = \ep = 2\times 10^{-3}$, $f = 1$, $g=u_0=0$, and take the background domain as $\wQ = (0,1) \times (0,1)$. 
The initial mesh $\Th$ is generated following the procedure in Example~\ref{ex2} with  $N_x=N_t=7$, and is uniformly refined four times to obtain the computational mesh $\mathcal{T}_{h/16}$. Since \(\varepsilon\) is smaller than the characteristic mesh size, the solution exhibits characteristics analogous to those observed in convection-dominated convection-diffusion problems.

The parameters are chosen as $\ga_1=0.1$, $\ga = 20$ and $\de = \de_c/\ep$, where $\de_c = 0.3$ corresponds to the stabilized case and  $\de_c = 0$ to the unstabilized case. The numerical solution is computed on \(\mathcal{T}_{h/16}\) for both scenarios and plotted in Figure~\ref{boundary_layer}.

As shown in the figure, without stabilization the numerical solution exhibits spurious oscillations near the boundary. In contrast, when SUPG stabilization is applied, these oscillations are effectively suppressed, resulting in a smooth and physically consistent solution profile.

\begin{figure}[htbp]
	\makebox[\textwidth][c]
	{\subfigure{\includegraphics[width = 0.45\textwidth]{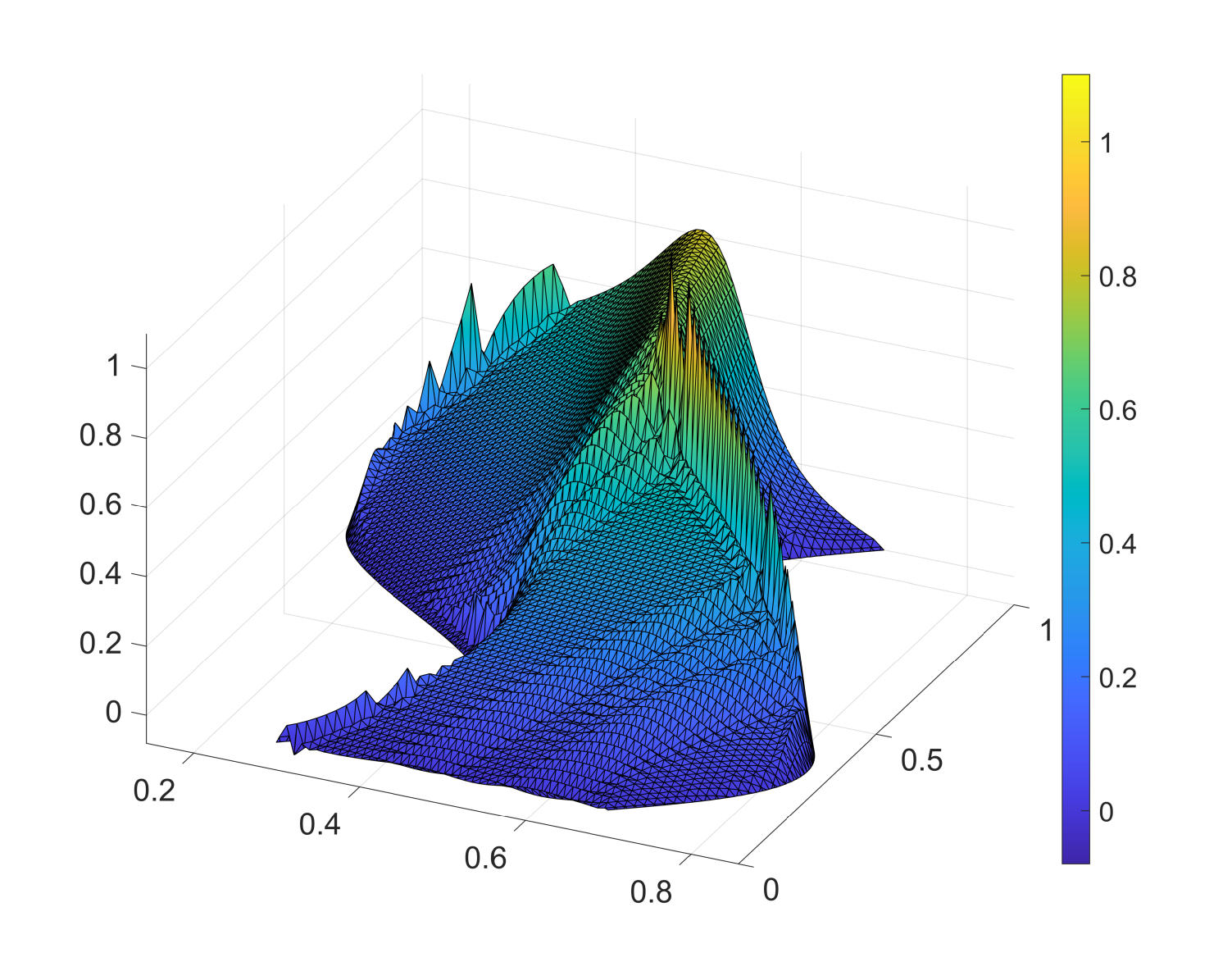}} 
	\subfigure{\includegraphics[width = 0.45\textwidth]{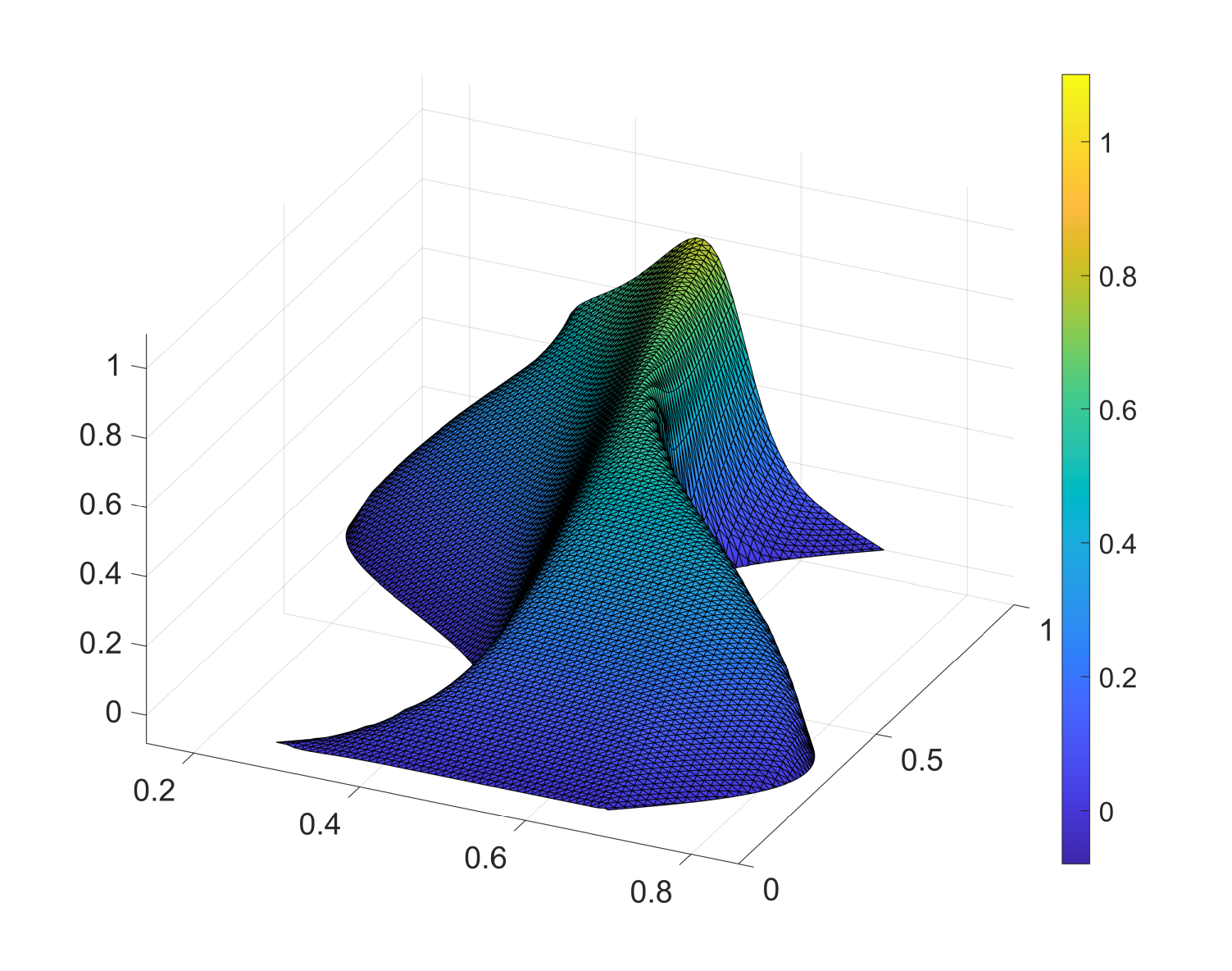}}} 
	\caption{Left: The method without stabilization ($\de_c = 0$). Right: The method with stabilization($\de_c = 0.3$).}
     \label{boundary_layer}
\end{figure} 

\section{Conclusions}\label{Sec7}
In this work, we have presented a novel unfitted space-time FEM designed for parabolic problems on moving domains, utilizing unstructured meshes. The key contributions of this study can be summarized as follows.

To ensure the robustness of the proposed method, two stabilization techniques were introduced. First, the ghost penalty stabilization was employed to effectively overcome the small cut problem, thereby controlling the condition number of the stiffness matrix. Second, the streamline upwind Petrov–Galerkin (SUPG) scheme was incorporated to stabilize the temporal advection term.

From a theoretical perspective, we established a rigorous a priori error estimate in a discrete energy norm, demonstrating that our method achieves the optimal convergence rate with respect to the mesh size. Furthermore, a space-time Poincaré–Friedrichs inequality was proved, which is essential for the condition number analysis of the overall formulation.

Numerical experiments in one and two spatial dimensions $(d = 1, 2)$ confirmed these theoretical findings. The results successfully verified the optimal convergence rate for the relative $H^{1,0}$ error and the optimal growth rate for the condition number of the stiffness matrix. Additional tests clearly demonstrated the crucial role of both the ghost penalty and SUPG stabilizations in maintaining numerical stability.

For future work, we plan to focus on the development of adaptive algorithms based on the proposed method and explore its potential extensions to more complex parabolic moving interface problems.

\bibliography{ref.bib}

\end{document}